\documentclass[11pt]{amsart}

\usepackage[margin=1in]{geometry}
\usepackage{amsmath,amssymb,amsthm,mathtools}
\usepackage{enumitem}
\usepackage[colorlinks=true,linkcolor=blue,citecolor=blue,urlcolor=blue]{hyperref}

\newtheorem{theorem}{Theorem}[section]
\newtheorem{maintheorem}[theorem]{Theorem}
\newtheorem{proposition}[theorem]{Proposition}
\newtheorem{lemma}[theorem]{Lemma}
\newtheorem{corollary}[theorem]{Corollary}
\theoremstyle{definition}
\newtheorem{definition}[theorem]{Definition}

\theoremstyle{remark}
\newtheorem{remark}[theorem]{Remark}

\newcommand{\Leb}{\operatorname{Leb}}
\newcommand{\Acc}{\operatorname{Acc}}
\newcommand{\Cov}{\operatorname{Cov}}
\newcommand{\Var}{\operatorname{Var}}
\newcommand{\eps}{\varepsilon}
\newcommand{\dd}{\,d}
\newcommand{\N}{\mathbb N}

\newcommand{\weakstar}{\stackrel{*}{\rightharpoonup}}
\numberwithin{equation}{section}

\title[Sparse empirical measures in the non-statistical regime]
{Sparse empirical measures for non-statistical interval maps with two neutral fixed points}

\author{Farrukh Mukhamedov}
\address{Department of Mathematical Sciences, College of Science,
United Arab Emirates University, Al Ain, United Arab Emirates}
\email{far75m@yandex.ru; farrukh.m@uaeu.ac.ae}
\urladdr{https://orcid.org/0000-0001-5728-6394}
\date{}

\hypersetup{
  pdftitle={Sparse empirical measures for non-statistical interval maps with two neutral fixed points},
  pdfauthor={Farrukh Mukhamedov}
}


\begin{document}

\begin{abstract}
We consider balanced doubly intermittent full-branch interval maps with two
neutral fixed points.  Although the ordinary empirical measures fail to
converge for Lebesgue-almost every initial point, we show that this
non-statistical behavior persists under a broad class of deterministic
observation schemes: if the sampling sequence is block-dominating, then both
endpoint Dirac masses are accumulation points of the sparse empirical
measures for Lebesgue-almost every orbit.  This class includes sampling
sequences of positive lower density, arithmetic and polynomial sequences,
and finite-index regularly varying sequences.  We also prove that every map
in the balanced class has the strong natural-measure property, including at
the boundary exponent $\beta=1$: the pushforwards of every absolutely
continuous initial probability converge to a distinguished endpoint mixture
$\nu_*$.  Abstract consequences of this one-time limit include annealed
convergence along every deterministic sampling sequence, the existence of
arbitrarily sparse sequences with almost-sure convergence, and an
almost-sure convergence criterion based on a power-saving variance bound.
For the symmetric quadratic map $g_1\in\mathfrak F_*$, we verify the
dyadic covariance condition on a countable uniformly dense endpoint-flat
family: the
double covariance sum is $O_\phi(n\log(n+2))$.  Consequently its dyadic
empirical measures converge for Lebesgue-almost every initial point to
$\tfrac12(\delta_{-1}+\delta_1)$.
\end{abstract}

\maketitle
\enlargethispage{2pt}

\section{Introduction}
\label{sec:introduction}

For a Borel measurable map $T:X\to X$ of a compact metric space $X$, the statistical
behavior of an orbit is usually described by the Birkhoff averages
\[
  \frac1n\sum_{j=0}^{n-1}\varphi(T^j x),
  \qquad \varphi\in C(X).
\]
Following Ruelle and Takens \cite{Rue01,Tak08}, an orbit is said to have
\emph{historic behavior} if this average fails to converge for at least one
continuous observable.  Equivalently, the ordinary empirical measures
\[
  \frac1n\sum_{j=0}^{n-1}\delta_{T^j x}
\]
fail to converge weak-star. In fact, compactness of $\mathcal P(X)$ in the
weak-star topology gives subsequential limits of the empirical measures, and
distinct limits are separated by a continuous function.  Conversely,
convergence of the empirical measures gives convergence of every continuous
Birkhoff average.  Birkhoff's ergodic theorem \cite{Bir31}, together with
separability of $C(X)$, implies that the historic set has measure zero for
every $T$-invariant Borel probability measure.  Historic behavior can
nevertheless be large relative to other notions of size.

We notice that the classical Bowen-type heteroclinic examples provide an early paradigm:
successively longer visits near different saddle points can force the time
averages to oscillate \cite{Gau92,Tak94}.  Takens later asked whether historic
behavior could occur on a positive-Lebesgue-measure set in a persistent
smooth setting \cite{Tak08}.  Kiriki and Soma in \cite{KS17} answered a version of this
problem by proving that, in every Newhouse open set of $C^r$ surface
diffeomorphisms, $2\le r<\infty$, maps with a nontrivial wandering domain whose
points have historic behavior form a dense subset.  Thus the
conclusion is density within each such open set, rather than persistence for
all its maps.  In a complementary direction, irregular sets can carry
full topological entropy and full Hausdorff dimension in several uniformly
hyperbolic and conformal settings \cite{BS00}; the saturated-set formalism of
Pfister and Sullivan gives a broad entropy framework for prescribed compact
connected sets of invariant accumulation measures under their
$g$-almost-product and uniform-separation hypotheses \cite{PS07}. In \cite{JM24} stochastic dynamical systems with
historical behavior have been investigated.

A distinct ensemble-level precursor is due to Hofbauer and Keller, who
constructed quadratic maps for which the Ces\`aro averages
$N^{-1}\sum_{j=0}^{N-1}(f^j)_*\Leb$ have no weak-star limit \cite{HK90}.
Their asymptotic-measure problem is different both from orbitwise historic
behavior and from the direct one-time convergence of pushforwards studied
below.

Intermittent interval maps with competing indifferent regions can realize a
different orbitwise mechanism.  The oscillation is then produced by
heavy-tailed laminar sojourns rather than by a heteroclinic cycle.  Limit laws
for long sojourns near indifferent fixed points and for occupation times of
infinite-measure components were developed by Thaler and by
Aaronson--Thaler--Zweim\"uller \cite{Tha02,ATZ05}.  These works provide
foundational results for the residence-time mechanism relevant to the maps
considered below.

Let $I=[-1,1]$ and let $g:I\to I$ be a balanced doubly intermittent map in
the class $\mathfrak F_*$ introduced in \cite{CLM23,CL24}.  The endpoints
$-1$ and $1$ are neutral fixed points and act as competing traps.  Coates and
Luzzatto proved that, for Lebesgue almost every $x$, the ordinary empirical
measures
\[
  \mu_n(x)=\frac1n\sum_{r=0}^{n-1}\delta_{g^r x}
\]
do not converge.  More precisely, their weak-star accumulation set is
\[
  \Omega=\{p\delta_1+(1-p)\delta_{-1}:p\in[0,1]\}.
\]
Thus almost every orbit has historic behavior, and the conclusion is sharper
than nonconvergence alone: every endpoint mixture occurs along a subsequence
of ordinary times.  There is no conflict with Birkhoff's theorem.  Lebesgue
measure is not invariant, while the invariant measure equivalent to
Lebesgue for this balanced regime is sigma-finite and infinite
\cite{CLM23,CL24}.

A complementary dimension-theoretic result was obtained by Coates and Gelfert
\cite[Theorem~A, Corollary~2.5, and Example~8.3]{CG26}.  For a class of
Markov interval maps with $d\ge2$ equally sticky neutral fixed points and
common excursion-tail exponent $\alpha\in(0,1)$, they prove that the ordinary
full-time basin of every fixed-point mixture has full Hausdorff dimension.
More generally, every prescribed nonempty closed connected subset of the fixed-point
simplex occurs as the exact accumulation set of the ordinary empirical
measures on a set of full Hausdorff dimension.  Their Example~8.3 includes,
through the inducing construction of \cite[Section~4.2]{CMT26}, the
two-neutral-fixed-point maps with critical or singular preimages introduced
in \cite{CLM23}.  For the balanced class considered here, the overlap is the
range $\beta>1$, equivalently $\alpha=1/\beta\in(0,1)$; the boundary case
$\beta=1$ is outside the exponent range assumed in \cite{CG26}.  The two
works address complementary questions.  Coates--Gelfert keep the full
observation sequence $0,1,2,\ldots$ and quantify exceptional sets of initial
points by Hausdorff dimension.  We keep Lebesgue-typical initial points and
vary the deterministic observation sequence, distinguishing schedules that
preserve non-statistical behavior from schedules that recover the
distinguished natural mixture.

This paper studies what remains of this behavior when the orbit is observed only at a deterministic sequence $\mathbf{a}=\{a_n\}$ such that
\[
  0\le a_0<a_1<a_2<\cdots.
\]
The corresponding sparse empirical measures are
\[
  \mu_n^\mathbf{a}(x)=\frac1n\sum_{k=0}^{n-1}\delta_{g^{a_k}x}.
\]
These are not consecutive-time empirical measures, and their accumulation
points need not be $g$-invariant.  Accordingly, failure of $\mu_n^\mathbf{a}(x)$ to
converge is a sampled analogue of historic behavior, not historic behavior in
the usual full-time sense.  For every $\varphi\in C(I)$,
\[
  \int_I\varphi\,d\mu_n^\mathbf{a}(x)
  =\frac1n\sum_{k=0}^{n-1}\varphi(g^{a_k}x).
\]
In particular, if both $\delta_{-1}$ and $\delta_1$ are accumulation points,
then the sampled averages diverge for every continuous $\varphi$ with
$\varphi(-1)\ne\varphi(1)$.  Conversely, weak-star convergence of
$\mu_n^\mathbf{a}(x)$ is precisely simultaneous convergence of these sampled averages
for all continuous observables.

Deterministic subsequence averages have their own classical history.  In
the setting of an invertible probability-preserving transformation $T$,
Bourgain proved that, for an integer-coefficient polynomial
$P\in\mathbb Z[t]$, the averages
\[
  \frac1N\sum_{k=1}^{N} f\circ T^{P(k)}
\]
converge almost surely for every $f\in L^q$, $q>1$
\cite[Theorem~1]{Bou89}, whereas
strong sweeping-out phenomena show that lacunary sampling can behave very
differently \cite{ABJLRW96}.  Those theorems concern invariant probability
spaces and do not apply directly here: Lebesgue measure is noninvariant and
the relevant absolutely continuous invariant measure is infinite.  They do,
however, emphasize that sparsity by itself neither forces convergence nor
forces nonconvergence.  The results below instead exploit the particular
excursion geometry of the intermittent map and the one-time convergence of
its pushforwards.

The question is natural for two closely related reasons.  First,
non-statisticality in this setting is generated by a temporal mechanism:
typical trajectories alternate between long laminar excursions near the two
neutral fixed points, and some excursions become comparable with, or much
longer than, all the time elapsed before them.  Ordinary empirical measures
give equal weight to every iterate and therefore combine two distinct
features: the geometry of these excursion blocks and the distribution of the
observation times among them.  Passing to a deterministic subsequence
separates the two.  It tests whether nonconvergence is an intrinsic feature of
the orbit, independent of the manner in which it is observed, or whether it
is tied to the full-time averaging protocol.  This question is particularly
natural for intermittent dynamics, since omitting or over-representing a
single exceptionally long excursion can substantially change the empirical
balance between the two endpoints.

Second, sparse sampling reveals a statistical structure hidden beneath the
pathwise oscillation.  The ordinary empirical measures of almost every orbit
have the whole segment $\Omega$ as their accumulation set, whereas the
one-time pushforwards of every absolutely continuous ensemble converge to a
single distinguished mixture $\nu_*$.  The results below explain how these
apparently conflicting facts coexist.  Sampling schemes that retain the
dominant excursion blocks preserve non-statistical behavior; suitably chosen
sufficiently sparse schemes can make the selected observations weakly
correlated enough to recover $\nu_*$ almost surely; and every deterministic
scheme recovers $\nu_*$ after averaging over an absolutely continuous
ensemble.  Thus the issue is not merely whether empirical measures converge,
but which temporal scales and observation protocols reveal the trapping
mechanism and which reveal the underlying natural measure.  In this way,
sparse sampling gives a finer description of non-statistical dynamics by
distinguishing robust orbit-wise memory from statistical regularity that
becomes visible under temporal thinning or ensemble averaging.

The first result identifies a broad class of sampling sequences that necessarily preserve the endpoint oscillations.  Write
\[
  A(t)=\#\{k\ge0:a_k<t\},
  \qquad t>0.
\]
We call $\mathbf{a}$ block-dominating when
\[
  \lim_{\substack{R\to\infty\\ R\in\mathbb N}}
  \liminf_{\substack{N\to\infty\\ N\in\mathbb N}}
  \frac{A(RN)}{A(N)}=\infty.
\]
The long-excursion mechanism in \cite{CL24} produces, near each endpoint, blocks whose length is arbitrarily large relative to all elapsed time.  Block domination says exactly that, when such a block occurs, the sampled observations inside the block asymptotically dominate all earlier sampled observations.  Consequently both $\delta_1$ and $\delta_{-1}$ remain sparse accumulation points.

The second map-specific part of the paper concerns the strong natural-measure property
\[
  (g^n)_*\lambda\weakstar\nu_*
  \qquad\text{for every probability }\lambda\ll\Leb,
\]
where
\[
  \nu_*=p_*\delta_1+(1-p_*)\delta_{-1},
  \qquad
  p_*=\frac{C_+}{C_++C_-},
\]
and $C_\pm$ are the two endpoint-tail constants.  We prove this property for every map in $\mathfrak F_*$.  A point that requires care is the relation between the two inducing schemes used in \cite{CLM23,CL24} and \cite{CMT26}.  We give the exact cylinder decomposition, including the index shifts, which combines the finite-flight assumption with the full-return expansion of \cite{CLM23} to verify condition F2(2) of \cite{CMT26}.  For $\beta>1$, hypotheses H1--H4 then follow from Theorem~4.9 of \cite{CMT26}.  At the boundary $\beta=1$, we verify the standing assumptions and H1--H4 directly.  For exactness, we equip the saturated tower with an explicit probability measure equivalent to its invariant sigma-finite measure, verify the Markov, generating, Radon--Nikodym Jacobian, R\'enyi-distortion, conservativity, and aperiodicity hypotheses of the Markov-fibred-system criterion \cite[Theorem~3.2]{ADU93}, and then descend exactness through the factor map.  We also reconcile the endpoint conventions and derive the tail constants from Lemma~4.11 of \cite{CMT26} by direct summation.

These two contributions---persistence under block-dominating sampling and the strong natural-measure theorem for the balanced class---are the results that use the specific intermittent dynamics.  The later annealed, variance, covariance, and diagonal-selection results are abstract consequences of one-time convergence and elementary second-moment arguments.  First, for every deterministic sampling sequence, the averaged pushforwards
\[
  \frac1n\sum_{k=0}^{n-1}(g^{a_k})_*\lambda
\]
converge to $\nu_*$.  Second, one can select observation times, as sparse as desired, so that the corresponding sparse empirical measures converge to $\nu_*$ for almost every initial point.  The construction uses the correlation form of the strong natural-measure property and chooses each new observation time so that finitely many previously selected observables have summably small covariances with the new one.

For a fixed prescribed sequence, one-time convergence is not enough.  A
standard sufficient route to almost-sure convergence is a power-saving
variance bound, which in turn can be verified by a double covariance
estimate.  We prove a general criterion of this form and state its dyadic
specialization.  A one-time renewal estimate does not by itself yield the
uniform two-time estimate required after conditioning at an earlier
observation time.  For the symmetric quadratic map $g_1$, however, symmetry
and the exact quadratic endpoint law permit a complete-cylinder
decomposition.  In Appendix~\ref{app:g1-dyadic}, we combine the exponent-one
operator-renewal theorem of Melbourne--Terhesiu \cite{MT12} with a separate
estimate on the odd renewal channel to obtain an $O(n\log n)$ dyadic
covariance bound on a countable uniformly dense endpoint-flat family of
observables.  Thus the criterion is verified for $g_1$, while its
verification for every map in $\mathfrak F_*$ remains open.

The paper is organized as follows.  Section~\ref{sec:setup} gives the
definitions and inducing input, Section~\ref{sec:block} proves persistence
under block-dominating sampling, Section~\ref{sec:natural} proves the strong
natural-measure result and its annealed consequence, and
Section~\ref{sec:strong-law} develops the abstract variance and selection
criteria.  Appendix~\ref{app:g1-dyadic} verifies the dyadic criterion for
$g_1$.

\section{Setup and inducing facts}
\label{sec:setup}

\subsection{Empirical measures and notation}

Throughout,
\[
  I=[-1,1],
  \qquad
  m=\frac{\Leb|_I}{\Leb(I)}
\]
denotes normalized Lebesgue measure.  We use
$\N=\{1,2,\ldots\}$ and $\N_0=\{0,1,2,\ldots\}$.  For nonnegative
quantities $u$ and $v$, the notation $u\lesssim v$ means that
$u\le Cv$ for a constant $C$ independent of the asymptotic indices under
consideration.  Weak-star convergence of Borel probability measures on $I$ is denoted by $\weakstar$, and $C(I)$ denotes the real-valued continuous functions on $I$.  For a measurable map $f:I\to I$, define
\[
  \mu_n^f(x)=\frac1n\sum_{r=0}^{n-1}\delta_{f^r x}.
\]
When $f=g$, we simply write $\mu_n(x)$.

\begin{definition}
An invariant Borel probability measure $\nu$ is a \textit{physical measure} for a measurable map $f:I\to I$ if its basin
\[
  \mathcal B_f(\nu)
  =\left\{x\in I:\frac1n\sum_{r=0}^{n-1}\delta_{f^r x}\weakstar\nu\right\}
\]
has positive $m$-measure.
\end{definition}

Let $\mathbf{a}=(a_k)_{k\ge0}$ be a strictly increasing sequence of nonnegative integers.  Define
\[
  \mu_n^\mathbf{a}(x)=\frac1n\sum_{k=0}^{n-1}\delta_{g^{a_k}x}
\]
and
\[
  X_\mathbf{a}=\{x\in I:\mu_n^\mathbf{a}(x)\text{ does not converge weak-star}\}.
\]
For arithmetic sampling $a_k=qk$, we write $X_{\mathrm{arith},q}$.  For dyadic sampling $d_k=2^k$, we write
\[
  \mu_n^{\mathrm{dyad}}(x)
  =\frac1n\sum_{k=0}^{n-1}\delta_{g^{2^k}x},
  \qquad
  X_{\mathrm{dyad}}=X_\mathbf{d}, \ \ \mathbf{d}=\{d_k\}.
\]
This notation keeps arithmetic sampling with step $2$ distinct from dyadic sampling.

For $\eps>0$, let
\[
  U_\eps^+=(1-\eps,1],
  \qquad
  U_\eps^-=[-1,-1+\eps).
\]
Adding the endpoints to the open neighborhoods used in \cite{CL24} does not alter any almost-everywhere statement.

\subsection{The balanced map class}

We give the definition of the class used in this paper.  This also fixes the
translation between the sign notation used below and the indexed notation in
\cite{CLM23,CL24}.  Write
\[
  I_-=[-1,0],\qquad I_+=[0,1],
\]
and denote the two open branches of a map $g:I\to I$ by $g_-$ and $g_+$,
also using these symbols for their one-sided continuous extensions to $0$.
Thus $g_-(0)=1$ and $g_+(0)=-1$ in the local model.  We fix one branch
value as the convention for $g(0)$ (for instance, $g(0)=1$).  Any other
convention that does not create an additional fixed point changes only the
forward orbit of the discontinuity and its countable backward orbit, a
Lebesgue-null set.  The
relevant maps have two neutral endpoints, so their defining conditions are as follows:

\medskip
\noindent\textbf{(A0) Full-branch structure.}
The restrictions
\[
  g_-:\mathring I_-\longrightarrow(-1,1),
  \qquad
  g_+:\mathring I_+\longrightarrow(-1,1)
\]
are orientation-preserving $C^2$ diffeomorphisms, and the only fixed points
of $g$ are $-1$ and $1$.

\medskip
\noindent\textbf{(A1) Local forms.}
There are constants
\[
  \ell_-,\ell_+,\iota,k_-,k_+,a_-,a_+,b_-,b_+>0
\]
such that, with
\[
  U_{0-}=(-\iota,0],\qquad U_{0+}=[0,\iota),
  \qquad U_{-1}=g_+(U_{0+}),\qquad U_{+1}=g_-(U_{0-}),
\]
the endpoint branches satisfy
\begin{equation}
  g_-(x)=x+b_-(1+x)^{1+\ell_-}\quad(x\in U_{-1}),
  \qquad
  g_+(x)=x-b_+(1-x)^{1+\ell_+}\quad(x\in U_{+1}).
  \label{eq:A1-endpoints}
\end{equation}
At the two sides of the discontinuity,
\begin{equation}
  g_-(x)=1-a_-|x|^{k_-}\quad(x\in U_{0-})
  \quad\text{if }k_-\ne1,
  \qquad
  g_+(x)=-1+a_+x^{k_+}\quad(x\in U_{0+})
  \quad\text{if }k_+\ne1.
  \label{eq:A1-critical}
\end{equation}
If $k_-=1$ or $k_+=1$, the corresponding formula in
\eqref{eq:A1-critical} is replaced by
\[
  g_-'(0)=g'(0-)=a_->1
  \qquad\text{or}\qquad
  g_+'(0)=g'(0+)=a_+>1,
\]
respectively, together with monotonicity on the corresponding one-sided
neighborhood.  This is the one-sided derivative convention used throughout.

\medskip
\noindent\textbf{(A2) Finite-flight expansion.}
Define
\[
  \Delta_0^-=g_-^{-1}((0,1)),
  \qquad
  \Delta_0^+=g_+^{-1}((-1,0)),
\]
and, for $n\ge1$,
\[
  \Delta_n^-=g_-^{-1}(\Delta_{n-1}^-),
  \qquad
  \Delta_n^+=g_+^{-1}(\Delta_{n-1}^+),
\]
\[
  \delta_n^-=g_-^{-1}(\Delta_{n-1}^+)\cap\Delta_0^-,
  \qquad
  \delta_n^+=g_+^{-1}(\Delta_{n-1}^-)\cap\Delta_0^+.
\]
If $\Delta_0^\pm\subset U_{0\pm}$, set $n_\pm=0$; otherwise set
\[
  n_\pm=\min\{n\ge1:\delta_n^\pm\subset U_{0\pm}\}.
\]
There is a constant $\lambda_A>1$ such that, for either sign,
\begin{equation}
  (g^n)'(x)>\lambda_A
  \qquad
  (1\le n\le n_\pm,\ x\in\delta_n^\pm).
  \label{eq:A2-expansion}
\end{equation}
When $n_\pm=0$, the corresponding requirement is empty.

In the notation of \cite{CLM23,CL24},
\[
  (\ell_1,\ell_2,k_1,k_2,a_1,a_2,b_1,b_2)
  =(\ell_-,\ell_+,k_-,k_+,a_-,a_+,b_-,b_+).
\]
The two stickiness parameters are
\begin{equation}
  \beta^-=\ell_-k_+,
  \qquad
  \beta^+=\ell_+k_-,
  \qquad
  \beta=\max\{\beta^-,\beta^+\}.
  \label{eq:stickiness}
\end{equation}
The balanced non-statistical class used in this paper is therefore
\begin{equation}
  \mathfrak F_*
  =\left\{g:\text{$g$ satisfies \textup{(A0)--(A2)}, }
      \beta\ge1,\ \beta^-=\beta^+\right\}.
  \label{eq:def-Fstar}
\end{equation}
This is exactly the class denoted by $\mathfrak F_*$ in \cite{CL24}.
The larger ambient class in \cite{CLM23} also permits an expanding endpoint
with $\ell_-=0$ or $\ell_+=0$ and imposes the corresponding relaxed endpoint
regularity clause.  Those alternatives cannot occur in \eqref{eq:def-Fstar}:
indeed, $\beta^-=\beta^+\ge1$ forces both $\ell_->0$ and $\ell_+>0$.
Thus \textup{(A0)--(A2)} above are the complete defining conditions needed
here, not a reduced substitute for the boundary cases of $\mathfrak F_*$.

In the sequel, we always assume that
\begin{equation}
  g\in\mathfrak F_*,
  \qquad
  \beta^- =\beta^+=:\beta\ge1,
  \qquad
  \alpha=\frac1\beta\in(0,1].
  \label{eq:balance}
\end{equation}

The first-return construction in \cite{CLM23,CL24} provides two inducing
bases, one on each side of the discontinuity.  The associated first-return
maps are full-branch Gibbs--Markov maps and have invariant probability
measures equivalent to Lebesgue on their bases.  To define the tail constants
used throughout, we fix the left inducing base $\Delta_0^-$ and denote its
invariant probability by $\widehat\mu=\widehat\mu_-$.  We denote the
corresponding invariant probability on the right inducing base by
$\widehat\mu_+$.  If $\tau^+$ and $\tau^-$ denote,
respectively, the lengths of the excursions toward $1$ and $-1$ in one return
cycle from this base, then
\begin{equation}
  \widehat\mu(\tau^+>n)\sim C_+n^{-\alpha},
  \qquad
  \widehat\mu(\tau^->n)\sim C_-n^{-\alpha},
  \label{eq:tails}
\end{equation}
with $C_+,C_->0$; see \cite[Proposition~3.1]{CL24}.  Set
\begin{equation}
  p_*=\frac{C_+}{C_++C_-},
  \qquad
  \nu_*=p_*\delta_1+(1-p_*)\delta_{-1}.
  \label{eq:nustar}
\end{equation}

\begin{theorem}\cite{CL24}
\label{thm:CL}
For $m$-almost every $x\in I$,
\[
  \Acc\{\mu_n(x):n\ge1\}
  =\{p\delta_1+(1-p)\delta_{-1}:p\in[0,1]\}.
\]
In particular, $g$ is non-statistical and has no physical measure.
\end{theorem}

\begin{proof}
This is \cite[Theorem~3]{CL24}.  The absence of a physical measure follows because the empirical measures fail to converge on a set of full Lebesgue measure.
\end{proof}

\begin{remark}
\label{rem:Coates-Gelfert}
Suppose that $\beta>1$ and write
$\nu_p=p\delta_1+(1-p)\delta_{-1}$ for $p\in[0,1]$.  After the
harmless affine normalization of the interval, Theorem~A and
Corollary~2.5 of \cite{CG26}, together with Example~8.3 there, give
\[
  \dim_{\mathrm H}\mathcal B_g(\nu_p)=1
  \qquad(p\in[0,1]).
\]
More generally, for every nonempty closed connected set $C\subset[0,1]$,
\[
  \dim_{\mathrm H}
  \left\{x\in I:
  \Acc\{\mu_n(x):n\ge1\}=\{\nu_p:p\in C\}\right\}=1.
\]
Here the restriction $\beta>1$ is necessary for this citation, since the
common excursion-tail exponent in \cite{CG26} is required to belong to
$(0,1)$, whereas \eqref{eq:balance} gives $\alpha=1/\beta$.  On the other
hand, Theorem~\ref{thm:CL} implies
$m(\mathcal B_g(\nu_p))=0$ for every $p\in[0,1]$.  Thus these ordinary-time
basins are Lebesgue-null but have full Hausdorff dimension.  This comparison
is not used in the proofs below.
\end{remark}

The proof of Theorem~\ref{thm:CL} contains a stronger block statement that is the input needed below.  We record it with the two inducing bases made explicit.

\begin{proposition}
\label{prop:excursions}
There is a set $I_0\subset I$ with $m(I_0)=1$ such that the following holds.  For every $x\in I_0$ and every $\sigma\in\{+,-\}$ there are integers $T_j^\sigma\to\infty$ and $L_j^\sigma\ge1$ satisfying
\begin{equation}
  \frac{L_j^\sigma}{T_j^\sigma}\longrightarrow\infty.
  \label{eq:dom-ratio}
\end{equation}
Moreover, for every sufficiently small $\eps>0$, there is a constant $B_{\eps,\sigma}<\infty$ such that
\begin{equation}
  \#\left\{T_j^\sigma\le r<T_j^\sigma+L_j^\sigma:
       g^r x\notin U_\eps^\sigma\right\}
  \le B_{\eps,\sigma}
  \label{eq:block-good}
\end{equation}
for all $j$.
\end{proposition}

\begin{proof}
We identify the full-measure sets on the two inducing bases and then transfer them to the whole interval.

For the right endpoint, let
\[
  G_-:\Delta_0^-\to\Delta_0^- ,\qquad
  \varphi_-:\Delta_0^-\to\N
\]
be the first-return map and its return time, and put
\[
  S_k^-(y)=\sum_{r=0}^{k-1}\varphi_-(G_-^r y).
\]
Let $\tau_+^-(y)$ be the length of the first, right-going leg of this return cycle.  The argument proving \cite[equations~(32)--(33) and Proposition~5.4]{CL24} gives a set
\[
  E_+^-=
  \left\{y\in\Delta_0^-:
  \limsup_{k\to\infty}
  \frac{\tau_+^-(G_-^k y)}{S_k^-(y)}=\infty\right\}
\]
of full $\widehat\mu_-$-measure, hence of full Lebesgue measure in $\Delta_0^-$ because the Gibbs--Markov invariant density is bounded above and below there.  For $y\in E_+^-$ choose $k_j\to\infty$ so that the displayed quotient tends to infinity and set
\[
  T_j^+(y)=S_{k_j}^-(y)+1,
  \qquad
  L_j^+(y)=\tau_+^-(G_-^{k_j}y).
\]
Then $T_j^+(y)\to\infty$ and, by construction,
\[
  \frac{L_j^+(y)}{T_j^+(y)}
  =\frac{\tau_+^-(G_-^{k_j}y)}{S_{k_j}^-(y)+1}
  \longrightarrow\infty.
\]
The interval of times
$[T_j^+(y),T_j^+(y)+L_j^+(y))$ is exactly the first, right-going
leg of the selected return cycle.  The added $1$ agrees with the convention in
\cite{CL24}, where an excursion is counted from the first iterate after a
return to the base.

The bound outside an endpoint neighborhood is deterministic and uniform in the selected cycle.  In the notation of the proof of \cite[Proposition~4.4]{CL24}, let $N_{\eps,+}$ be the last right-hand level not wholly contained in $U_\eps^+$.  Equation~(27) there says that a complete return cycle based at $u$ contains at least
$\tau_+^-(u)-N_{\eps,+}-1$ visits to $U_\eps^+$.  For sufficiently small
$\eps$, we have $U_\eps^+\subset I_+$, and the $\tau_+^-(u)$ indices
$1\le r\le\varphi_-(u)$ for which $g^ru\in I_+$ are consecutive and form
precisely $[1,1+\tau_+^-(u))$.  Hence all these visits occur in the selected
block.  Consequently
\[
  \#\{T_j^+(y)\le r<T_j^+(y)+L_j^+(y):g^r y\notin U_\eps^+\}
  \le N_{\eps,+}+1,
\]
where the right-hand side depends only on $\eps$ and the endpoint geometry, not on $j$ or on the chosen excursion.

For the left endpoint, use the symmetric first-return construction
$G_+:\Delta_0^+\to\Delta_0^+$ and write
\[
  \varphi_+:\Delta_0^+\to\N,
  \qquad
  S_k^+(y)=\sum_{r=0}^{k-1}\varphi_+(G_+^r y)
\]
for its return time and successive return sums; see
\cite[Remark~2.3 and Section~3.1]{CLM23}.  Applying the proof of
\cite[Section~5]{CL24} with the signs interchanged gives a set
\[
  E_-^+=
  \left\{y\in\Delta_0^+:
  \limsup_{k\to\infty}
  \frac{\tau_-^+(G_+^k y)}{S_k^+(y)}=\infty\right\},
\]
of full $\widehat\mu_+$-measure, hence of full Lebesgue measure in
$\Delta_0^+$.  Here $\tau_-^+$ is the first, left-going leg.  Choose
$k_j\to\infty$ along which the displayed quotient tends to infinity and put
\[
  T_j^-(y)=S_{k_j}^+(y)+1,
  \qquad
  L_j^-(y)=\tau_-^+(G_+^{k_j}y).
\]
The same calculation gives $L_j^-(y)/T_j^-(y)\to\infty$.  If
$N_{\eps,-}$ denotes the sign-reversed left-hand level cutoff, the
sign-reversed excursion estimate gives the uniform bound
$B_{\eps,-}=N_{\eps,-}+1$.

It remains to pass from the bases to almost every point of $I$.  Condition
\textup{(A0)} and the absence of interior fixed points imply
$g_-(x)>x$ on $(-1,0)$ and $g_+(x)<x$ on $(0,1)$.  Consequently every
point outside the endpoints and the countable backward orbit of the
discontinuity enters each of the two bases in finite time: while an orbit
remains on one side it is monotone, and failure to reach the crossing interval
$\Delta_0^\pm$ would force convergence to an interior fixed point.  Thus, for
each sign $\sigma$, almost every $x\in I$ has a finite first entrance time
\[
  e_\sigma(x)=\min\{r\ge0:g^r x\in\Delta_0^{-\sigma}\},
\]
where $\Delta_0^{-\sigma}$ denotes $\Delta_0^-$ for $\sigma=+$ and $\Delta_0^+$ for $\sigma=-$.  Since $g$ is nonsingular and the complement of $E_+^-$ or $E_-^+$ in the corresponding base is null, the set of $x$ whose first entrance lands outside the relevant $E$ is contained in a countable union of null preimages and is therefore null.  For every remaining $x$, put $y=g^{e_\sigma(x)}x$ and translate the blocks by
\[
  T_j^\sigma(x)=e_\sigma(x)+T_j^\sigma(y),
  \qquad
  L_j^\sigma(x)=L_j^\sigma(y).
\]
The fixed entrance time does not affect $L_j^\sigma/T_j^\sigma\to\infty$, and the orbit segment, hence the uniform bad-iterate bound, is unchanged.  Intersecting the two resulting full-measure sets proves the proposition.
\end{proof}

\section{Block-dominating sampling}
\label{sec:block}

For an increasing sampling sequence $\mathbf{a}=(a_k)$ and a real number $t>0$, let
\[
  A(t)=\#\{k\ge0:a_k<t\}.
\]
All limits involving $N$ and $R$ in this section are taken through positive integers.

\begin{definition}
The sequence $\mathbf{a}$ is block-dominating if
\begin{equation}
  \lim_{\substack{R\to\infty\\ R\in\mathbb N}}
  \liminf_{\substack{N\to\infty\\ N\in\mathbb N}}
  \frac{A(RN)}{A(N)}=\infty.
  \label{eq:BD}
\end{equation}
\end{definition}

\begin{maintheorem}
\label{thm:block}
Let $g\in\mathfrak F_*$ and let $\mathbf{a}$ be block-dominating.  Then, for $m$-almost every $x$,
\[
  \delta_1,\delta_{-1}\in
  \Acc\{\mu_n^\mathbf{a}(x):n\ge1\}.
\]
Moreover,
\[
  m(X_\mathbf{a})=1.
\]
\end{maintheorem}

\begin{proof}
Fix $x$ in the full-measure set of Proposition~\ref{prop:excursions}, fix $\sigma\in\{+,-\}$, and abbreviate
\[
  T_j=T_j^\sigma,
  \qquad
  L_j=L_j^\sigma,
  \qquad
  R_j=\frac{T_j+L_j}{T_j}.
\]
Then $R_j\to\infty$.

We first show
\begin{equation}
  \frac{A(T_j)}{A(T_j+L_j)}\longrightarrow0.
  \label{eq:count-ratio}
\end{equation}
Let $M>0$.  By \eqref{eq:BD}, there is a fixed $R>1$ such that
\[
  \frac{A(RN)}{A(N)}>M,
\]
for all sufficiently large $N$.

Hence, for all sufficiently large $j$, we have both $R_j\ge R$ and $T_j$ in this asymptotic range.  By monotonicity of $A$,
\[
  \frac{A(T_j+L_j)}{A(T_j)}
  \ge\frac{A(RT_j)}{A(T_j)}>M.
\]
Since $M$ is arbitrary, \eqref{eq:count-ratio} follows.

Put $q_j=A(T_j+L_j)$.  Since $T_j+L_j\to\infty$ and the sampling sequence is infinite, $q_j\to\infty$.  Passing to a subsequence and relabeling, we may and do assume that the integer sequence $(q_j)$ is strictly increasing.  Among the first $q_j$ sampled times, at most $A(T_j)$ occur before the excursion block.  By \eqref{eq:block-good}, at most $B_{\eps,\sigma}$ sampled times inside the block lie outside $U_\eps^\sigma$.  Therefore
\begin{equation}
  1-\mu_{q_j}^\mathbf{a}(x)(U_\eps^\sigma)
  \le
  \frac{A(T_j)+B_{\eps,\sigma}}{A(T_j+L_j)}
  \longrightarrow0.
  \label{eq:block-mass}
\end{equation}
The same subsequence $q_j$ works for every fixed sufficiently small $\eps$.
Let $e_\sigma=1$ for $\sigma=+$ and $e_\sigma=-1$ for $\sigma=-$.
For every $\phi\in C(I)$, \eqref{eq:block-mass} gives
\[
  \limsup_{j\to\infty}
  \left|\int_I\phi\dd\mu_{q_j}^\mathbf{a}(x)-\phi(e_\sigma)\right|
  \le
  \sup_{y\in U_\eps^\sigma}|\phi(y)-\phi(e_\sigma)|.
\]
The right-hand side tends to zero as $\eps\downarrow0$.  Hence
\[
  \mu_{q_j}^\mathbf{a}(x)\weakstar\delta_1
  \quad\text{when }\sigma=+,
  \qquad
  \mu_{q_j}^\mathbf{a}(x)\weakstar\delta_{-1}
  \quad\text{when }\sigma=-.
\]
Thus the sparse empirical measures have two distinct accumulation points for almost every $x$.
\end{proof}

\begin{corollary}
\label{cor:density}
Suppose
\[
  \underline d(\mathbf{a})=\liminf_{N\to\infty}\frac{A(N)}N>0.
\]
Then $\mathbf{a}$ is block-dominating and $m(X_\mathbf{a})=1$.  In particular,
\[
  m(X_{\mathrm{arith},q})=1
  \qquad\text{for every integer }q\ge1.
\]
\end{corollary}

\begin{proof}
Choose $c>0$ such that $A(N)\ge cN$ for all sufficiently large $N$.  Since $a_k$ are distinct nonnegative integers, $A(N)\le N$.  Hence, for fixed $R>1$ and all sufficiently large $N$,
\[
  \frac{A(RN)}{A(N)}\ge cR.
\]
Letting $R\to\infty$ we get \eqref{eq:BD}. Hence, by Theorem~\ref{thm:block} arrive at the required assertion.
\end{proof}

\begin{corollary}
\label{cor:RV}
Suppose that, after changing finitely many terms if necessary,
\[
  a_k\sim k^\rho L(k)
  \qquad(k\to\infty),
\]
where $\rho\in[1,\infty)$ and $L$ is slowly varying, and suppose $a_k$ is strictly increasing.  Then $\mathbf{a}$ is block-dominating and $m(X_\mathbf{a})=1$.
\end{corollary}

\begin{proof}
After the finite initial modification, extend the strictly increasing
sequence to a continuous, strictly increasing function $\widetilde a$ by
linear interpolation through the points $(k,a_k)$.  The uniform convergence
theorem for regularly varying sequences implies
\[
  \widetilde a(t)\sim t^\rho L(t),
  \qquad t\to\infty,
\]
so $\widetilde{\mathbf{a}}$ is regularly varying with index $\rho$.  Let
\[
  V(y)=\inf\{t:\widetilde a(t)\ge y\}
\]
be its generalized inverse.  By the generalized inverse theorem
\cite[Theorem~1.5.12]{BGT87}, $V$ is regularly varying with index
$1/\rho$.  If $a_k<N\le a_{k+1}$, then, up to the harmless choice of the
initial index, $A(N)=k+1$ while $V(N)\in(k,k+1]$.  Hence
\[
  |A(N)-V(N)|\le 1
\]
for all sufficiently large $N$.  This proves that our strict counting
convention $A(N)=\#\{k:a_k<N\}$ has the same regular variation as $V$.
Thus, for every fixed $R>0$,
\[
  \frac{A(RN)}{A(N)}\longrightarrow R^{1/\rho}.
\]
Again, sending $R\to\infty$ we obtain \eqref{eq:BD}, and Theorem~\ref{thm:block} applies.
\end{proof}

\begin{remark}
Corollary~\ref{cor:RV} includes $a_k=k^d$ for every integer $d\ge1$, and more generally $a_k=\lfloor k^\rho L(k)\rfloor$ whenever the resulting sequence is a strictly increasing sequence of nonnegative integers.
\end{remark}

\begin{remark}
For $d_k=2^k$,
\[
  A(N)=\log_2N+O(1),
\]
so $A(RN)/A(N)\to1$ for each fixed $R>1$.  Hence dyadic sampling is not block-dominating.  This observation alone gives no convergence conclusion.
\end{remark}

\section{Natural measures and annealed sparse convergence}
\label{sec:natural}

We first isolate the one-time statistical property used later.

\begin{definition}
\label{ass:SNM}
We say that $g$ has the \textit{strong natural-measure (SNM)} property if, for every Borel probability measure $\lambda\ll m$,
\begin{equation}
  (g^n)_*\lambda\weakstar\nu_*
  \qquad(n\to\infty),
  \label{eq:SNM}
\end{equation}
where $\nu_*$ is defined in \eqref{eq:nustar}.  The adjective ``strong'' emphasizes that the convergence is direct, without Ces\`aro averaging, and holds for every absolutely continuous initial probability.  In particular, it implies the common Ces\`aro-averaged formulation used for natural measures.
\end{definition}

\subsection{The abstract natural-measure theorem}

For the sake of completeness, we state a general abstract framework from \cite[Sections~2.1--2.2]{CMT26}
before applying it.  Let $X$ be a compact metric space with Borel probability
measure $m$, let $\mathcal B_X$ denote the Borel sigma-algebra, let
$\mathcal P$ be a countable measurable partition of $X$ into
positive-measure sets, and let $f:X\to X$ be nonsingular, meaning that
$m(f^{-1}E)=0$ whenever $m(E)=0$.  The pair
$(f,\mathcal P)$ is a \emph{Markov map} if every restriction $f|_a$,
$a\in\mathcal P$, is a measurable bijection from $a$ onto a union of elements
of $\mathcal P$.  Its separation time and symbolic metric are
\[
  s(x,x')=\inf\{n\ge0:f^nx\text{ and }f^nx'
                    \text{ lie in different atoms of }\mathcal P\},
  \qquad
  d_\theta(x,x')=\theta^{s(x,x')},
\]
where $\theta\in(0,1)$, and it is assumed that $s(x,x')=\infty$ only when
$x=x'$.  The map is topologically mixing in the symbolic topology if, for
every $a,b\in\mathcal P$, there is $N=N(a,b)$ such that
$f^na\cap b\ne\varnothing$ for all $n\ge N$.

Let $Y\subset X$ be a union of atoms of $\mathcal P$.  Define the first-return
time and map
\[
  \tau(y)=\inf\{n\ge1:f^ny\in Y\},
  \qquad
  F(y)=f^{\tau(y)}y.
\]
Under the conservativity assumption below, $\tau<\infty$ for $m$-almost every
$y\in Y$.  We discard the null set of nonreturning points, and all return-map
and return-partition statements are understood modulo this null set and
partition boundaries.
The return partition $\mathcal P_Y$ consists of the nonempty sets
\[
  [a_0,\ldots,a_{n-1};Y]
  =a_0\cap f^{-1}a_1\cap\cdots\cap f^{-(n-1)}a_{n-1}\cap f^{-n}Y,
\]
where $a_0\subset Y$ and $a_1,\ldots,a_{n-1}\subset X\setminus Y$.
The first-return map $F$ is \emph{Gibbs--Markov} when it is Markov with
respect to $\mathcal P_Y$, has finitely many images $Fa$, and there exist
$\theta\in(0,1)$ and $C>0$ such that
\[
  \left|\log JF(y)-\log JF(y')\right|
  \le C d_\theta(y,y')
  \qquad(y,y'\in a,\ a\in\mathcal P_Y),
\]
where, on each injective branch, $JF$ is the forward branch Jacobian defined by
\[
  m(FE)=\int_E JF\,dm
\]
for measurable $E$ in that branch.  This is the reciprocal of
$dm/d(m\circ F)$ used in the Gibbs--Markov bounded-distortion definition of
\cite{CMT26}; the absolute log-distortion conditions are equivalent, and the
transfer formulas below use $1/JF$.  We write
$\mathcal B_\theta(Y)$ for the real-valued functions that are Lipschitz with
respect to the symbolic metric associated to $\mathcal P_Y$.

The standing dynamical assumptions in \cite[Section~2.2]{CMT26} are:
\begin{enumerate}[label=\textup{(S\arabic*)},leftmargin=*,itemsep=2pt]
  \item $f$ is conservative, exact, and Markov.  Conservativity means that
  for every measurable $E$ of positive measure, almost every point of $E$
  returns to $E$ infinitely often.  Exactness means that
  \[
    \bigcap_{n\ge0}f^{-n}\mathcal B_X
  \]
  is trivial modulo $m$; in particular, exactness implies ergodicity.
  \item $F=f^\tau:Y\to Y$ is a Gibbs--Markov first-return map.
  \item Both $f$ and $F$ are topologically mixing in their symbolic
  topologies.
  \item The $f$-invariant absolutely continuous $\sigma$-finite measure
  supplied by \cite[Lemma~2.1]{CMT26} is unique up to scaling, equivalent to
  $m$, and is assumed to satisfy $\mu(X)=\infty$.
\end{enumerate}
Since exactness implies ergodicity, \textup{(S1)--(S2)} meet the hypotheses of
\cite[Lemma~2.1]{CMT26}.  That lemma supplies the invariant measure appearing
in \textup{(S4)} and shows that $0<\mu(Y)<\infty$, while the density of $\mu$
and its reciprocal are bounded and symbolic-Lipschitz on $Y$. Thus the formulation above records the standing assumptions in
\cite[Section~2.2]{CMT26}, including exactness, under which
\cite[Theorem~C]{CMT26} is stated.  Exactness is therefore not an
auxiliary assumption added by the present paper.

We now formulate the four additional hypotheses.  Fix an integer $d\ge1$;
$B_\eps(\xi)$ denotes the open metric ball of radius $\eps$ about $\xi$.

\begin{enumerate}[label=\textup{(H\arabic*)},leftmargin=*,itemsep=6pt]
  \item There are fixed points $\xi_1,\ldots,\xi_d\in X$ such that, for every
  $\eps>0$,
  \begin{equation}
    \mu\!\left(X\setminus\bigcup_{k=1}^dB_\eps(\xi_k)\right)<\infty.
    \label{eq:H1}
  \end{equation}

  \item There are $\alpha\in(0,1]$, constants $c_1,\ldots,c_d>0$, and a
  measurable partition $X_1,\ldots,X_d$ of $X\setminus Y$ such that:
  \begin{enumerate}[label=\textup{(\alph*)},leftmargin=2.2em,itemsep=2pt]
    \item $\xi_k\in\operatorname{Int}X_k$ for $k=1,\ldots,d$;
    \item an orbit cannot pass from $X_k$ to $X_\ell$, $k\ne\ell$, without
    first entering $Y$.  Equivalently, if
    \begin{equation}
      \tau^{(k)}(y)
      =\#\{1\le j\le\tau(y):f^jy\in X_k\},
      \label{eq:component-occupation}
    \end{equation}
    then at most one of the numbers $\tau^{(1)}(y),\ldots,\tau^{(d)}(y)$
    is nonzero and
    \begin{equation}
      \tau(y)=1+\sum_{k=1}^d\tau^{(k)}(y).
      \label{eq:H2-separation}
    \end{equation}
    In particular, for $n\ge2$,
    $\{\tau=n\}=\bigcup_{k=1}^d\{\tau^{(k)}=n-1\}$, while
    $\{\tau=1\}=\bigcap_{k=1}^d\{\tau^{(k)}=0\}$.  This separate
    $n=1$ clause removes the zero-occupation ambiguity in the shorthand union
    appearing in the statement of \textup{(H2)(b)} in \cite{CMT26};
    \item for $k=1,\ldots,d$,
    \begin{equation}
      \mu(\tau^{(k)}>n)\sim c_kn^{-\alpha}.
      \label{eq:H2-tail}
    \end{equation}
  \end{enumerate}

  \item To state this condition, for $r\ge0$ let $\mathcal P_r$ be the
  partition into nonempty $(r+1)$-cylinders
  \[
    [a_0,\ldots,a_r]=\bigcap_{j=0}^r f^{-j}a_j,
    \qquad a_j\in\mathcal P,
  \]
  and let $\mathcal P_r^*$ consist of those cylinders for which
  $a_0,\ldots,a_{r-1}\subset X\setminus Y$ and $a_r\subset Y$.
  For $\rho\in L^1(X,m)$ define
  \begin{equation}
    Q_r^\rho
    =\sum_{a\in\mathcal P_r^*}
       1_{f^ra}\left(\frac{\rho}{Jf^r}\right)
       \circ(f^r|_a)^{-1},
    \qquad
    Jf^r=\frac{d(m\circ f^r)}{dm}
    \quad\text{on }a.
    \label{eq:def-Qrho}
  \end{equation}
  There exist $\theta\in(0,1)$ and a dense set
  $\mathcal K(X)\subset L^1(X,m)$ such that
  \begin{equation}
    Q_r^\rho\in\mathcal B_\theta(Y)
    \qquad(\rho\in\mathcal K(X),\ r\ge0).
    \label{eq:H3}
  \end{equation}

  \item If $\alpha\in(0,\tfrac12]$, then
  \begin{equation}
    \mu(\tau=n)=O(n^{-1-\alpha}).
    \label{eq:H4}
  \end{equation}
\end{enumerate}

Under \textup{(H2)} put
\begin{equation}
  c_\tau=\sum_{k=1}^d c_k,
  \qquad
  \bar p_k=\frac{c_k}{c_\tau},
  \qquad
  \nu_{\bar p}=\sum_{k=1}^d\bar p_k\delta_{\xi_k}.
  \label{eq:CMT-limit-measure}
\end{equation}

\begin{theorem}\cite[Theorem C]{CMT26}
\label{thm:CMT-C}
Suppose that \textup{(S1)--(S4)} and hypotheses
\textup{(H1)--(H4)} hold with $\alpha\in(0,1]$.  Then
\[
  (f^n)_*\lambda\weakstar\nu_{\bar p}
  \qquad(n\to\infty)
\]
for every Borel probability measure $\lambda\ll m$.
\end{theorem}

\subsection{The comparison class used in the application theorem}

The concrete application in \cite[Section~4.2]{CMT26} uses a class that is
slightly more general at the neutral endpoints than the class in
\cite{CLM23}.  To avoid conflating it with $\mathfrak F_*$, denote it here by
$\mathcal F_{\mathrm{CMT}}$.  A map $f:I\to I$ belongs to
$\mathcal F_{\mathrm{CMT}}$ when it satisfies the following conditions.

\medskip
\noindent\textbf{(F0) Two full branches.}
For some $c\in(-1,1)$, the restrictions
\[
  f_-:(-1,c)\to(-1,1),
  \qquad
  f_+:(c,1)\to(-1,1)
\]
are orientation-preserving $C^2$ diffeomorphisms without fixed points.  We
use the normalization $c=0$.

\medskip
\noindent\textbf{(F1) Endpoint and discontinuity asymptotics.}
There are $\ell_\pm,k_\pm,a_\pm,b_\pm>0$ and one-sided neighborhoods
$U_{-1},U_{0-},U_{0+},U_{+1}$ such that, when $k_\pm\ne1$,
\begin{equation}
  f(x)=
  \begin{cases}
    x+b_-(1+x)^{1+\ell_-}+o((1+x)^{1+\ell_-}),&x\in U_{-1},\\
    1-a_-|x|^{k_-},&x\in U_{0-},\\
    -1+a_+x^{k_+},&x\in U_{0+},\\
    x-b_+(1-x)^{1+\ell_+}+o((1-x)^{1+\ell_+}),&x\in U_{+1}.
  \end{cases}
  \label{eq:F1}
\end{equation}
At $0$, the second and third lines of \eqref{eq:F1} refer to the one-sided
extensions $f_-$ and $f_+$, respectively; no condition is imposed on the
single assigned value $f(0)$.  If $k_-=1$ or $k_+=1$, the corresponding
discontinuity formula is replaced
by $f'(0-)=a_->1$ or $f'(0+)=a_+>1$, respectively.  The neighborhoods may,
and will, be shrunk so that $f(U_{0\pm})\subset U_{\mp1}$.

Given \textup{(F0)}, choose a cross-branch orbit of prime period two and denote its
points by $\gamma_-<0<\gamma_+$, so that
$f(\gamma_-)=\gamma_+$ and $f(\gamma_+)=\gamma_-$.  Set
\begin{equation}
  Y=[\gamma_-,\gamma_+],
  \qquad
  X_{\pm,n}=f_\pm^{-n}Y,
  \qquad
  Y_{\pm,n+1}=f_\pm^{-1}X_{\mp,n}
  \quad(n\ge1).
  \label{eq:CMT-partition}
\end{equation}
Then $\{X_{\pm,n}:n\ge1\}$ partitions $I\setminus Y$ and
$\{Y_{\pm,n}:n\ge2\}$ partitions $Y$, modulo endpoints; together they form
the Markov partition used by the application theorem.

\medskip
\noindent\textbf{(F2) Convexity, distortion, and finite-flight expansion.}
The map $f$ is convex on $U_{-1}$ and concave on $U_{+1}$.  In addition:
\begin{enumerate}[label=\textup{(\arabic*)},leftmargin=2.2em,itemsep=2pt]
  \item if the corresponding endpoint branch does not extend as a $C^2$
  map to the endpoint, then
  \[
    f''(x)\lesssim(1+x)^{\ell_--1}\quad(x\in U_{-1}),
    \qquad
    |f''(x)|\lesssim(1-x)^{\ell_+-1}\quad(x\in U_{+1});
  \]
  \item there is a common constant $\lambda_F>1$ such that, for each sign
  with $k_\pm\ne1$,
  \begin{equation}
    (f^n)'(x)>\lambda_F
    \quad\text{for all }x\in Y_{\pm,n}
    \text{ whenever }Y_{\pm,n}\not\subset U_{0\pm}.
    \label{eq:F2-expansion}
  \end{equation}
\end{enumerate}

\begin{lemma}
\label{lem:class-compatibility}
Every $g\in\mathfrak F_*$ belongs to $\mathcal F_{\mathrm{CMT}}$.  Under the
notation fixed above,
\[
  \alpha_+=\frac1{\ell_+k_-}=\frac1{\beta^+},
  \qquad
  \alpha_-=\frac1{\ell_-k_+}=\frac1{\beta^-}.
\]
In particular, balance is equivalent to $\alpha_+=\alpha_-=\alpha$.
\end{lemma}

\begin{proof}
Condition \textup{(A0)} is \textup{(F0)} with $c=0$.  The exact endpoint
identities in \eqref{eq:A1-endpoints} imply the endpoint asymptotics in
\textup{(F1)} with zero remainder, and the discontinuity clauses and the
$k_\pm=1$ conventions agree.  Differentiating
\eqref{eq:A1-endpoints} gives convexity on the left, concavity on the right,
and the bounds in \textup{(F2)(1)}.

We prove \textup{(F2)(2)} by comparing the two partitions directly.
The CLM intervals are defined by
\begin{equation}
  \Delta_r^\pm=g_\pm^{-r}(\Delta_0^\pm),
  \qquad
  \delta_r^-=g_-^{-1}(\Delta_{r-1}^+)\cap\Delta_0^-,
  \qquad
  \delta_r^+=g_+^{-1}(\Delta_{r-1}^-)\cap\Delta_0^+
  \quad(r\ge1),
  \label{eq:CLM-cylinder-definitions}
\end{equation}
with the inverse iterates taken along the indicated one-sided branch; this is
\cite[equations~(4)--(6)]{CLM23}.  Hence
\begin{equation}
  g^r:\delta_r^-\longrightarrow\Delta_0^+,
  \qquad
  g^r:\delta_r^+\longrightarrow\Delta_0^-
  \label{eq:CLM-half-return-images}
\end{equation}
bijectively.  For $m,n\ge1$, set
\begin{equation}
  \delta_{m,n}^-
    =\delta_m^-\cap g^{-m}(\delta_n^+),
  \qquad
  \delta_{m,n}^+
    =\delta_m^+\cap g^{-m}(\delta_n^-).
  \label{eq:CLM-full-return-branches}
\end{equation}
These are the branches of the full first-return maps $G_-$ and $G_+$, with
return time $m+n$.  Proposition~3.6 of \cite{CLM23}, together with its
sign-interchanged version, supplies $\lambda_R>1$ such that
\begin{equation}
  (g^{m+n})'(x)>\lambda_R
  \qquad
  (x\in\delta_{m,n}^\pm,\ m,n\ge1).
  \label{eq:full-return-expansion}
\end{equation}
Put $Y_-=Y\cap(-1,0)$ and $Y_+=Y\cap(0,1)$.  Since the branches are
increasing and $g(\gamma_-)=\gamma_+$, $g(\gamma_+)=\gamma_-$, the definitions
of $\Delta_0^\pm$ give, modulo endpoints,
\begin{equation}
  Y_-=Y\cap\Delta_0^-,
  \qquad
  Y_+=Y\cap\Delta_0^+.
  \label{eq:Y-halves-Delta0}
\end{equation}
For every $n\ge2$, the CMT definition
$Y_{-,n}=g_-^{-1}X_{+,n-1}$ implies that, for $x\in Y_{-,n}$,
\begin{equation}
  g^j x\in X_{+,n-j}\quad(1\le j\le n-1),
  \qquad g^n x\in Y.
  \label{eq:CMT-minus-itinerary}
\end{equation}
If $g^n x\in Y_+\subset\Delta_0^+$, then the recursive definition of the
$\Delta_r^+$ gives $gx\in\Delta_{n-1}^+$, and hence
$x\in\delta_n^-$.  Conversely, if $x\in Y_{-,n}\cap\delta_n^-$, then
\eqref{eq:CLM-half-return-images} and $g^n x\in Y$ give
$g^n x\in Y\cap\Delta_0^+=Y_+$.  Thus
\begin{equation}
  Y_{-,n}\cap g^{-n}Y_+
  =Y_{-,n}\cap\delta_n^-.
  \label{eq:cylinder-comparison-minus-a}
\end{equation}
If instead $g^n x\in Y_-$, then $g^{n-1}x\in\Delta_0^+$ and
\eqref{eq:CMT-minus-itinerary} gives $gx\in\Delta_{n-2}^+$; hence
$x\in\delta_{n-1}^-$.  Conversely, if
$x\in Y_{-,n}\cap\delta_{n-1}^-$, then
$g^{n-1}x\in\Delta_0^+$, so $g^n x\in(-1,0)$; together with
$g^n x\in Y$, this yields $g^n x\in Y_-$.  Therefore
\begin{equation}
  Y_{-,n}\cap g^{-n}Y_-
  =Y_{-,n}\cap\delta_{n-1}^-.
  \label{eq:cylinder-comparison-minus-b}
\end{equation}
The same argument with the signs interchanged gives
\begin{align}
  Y_{+,n}\cap g^{-n}Y_-
    &=Y_{+,n}\cap\delta_n^+,
    \label{eq:cylinder-comparison-plus-a}\\
  Y_{+,n}\cap g^{-n}Y_+
    &=Y_{+,n}\cap\delta_{n-1}^+.
    \label{eq:cylinder-comparison-plus-b}
\end{align}
Since $g^n(Y_{\pm,n})=Y_-\mathbin{\dot\cup}Y_+$ modulo endpoints,
Equations~\eqref{eq:cylinder-comparison-minus-a}--
\eqref{eq:cylinder-comparison-plus-b} give the disjoint decompositions
\begin{align}
  Y_{-,n}
    &=\bigl(Y_{-,n}\cap\delta_n^-\bigr)
      \mathbin{\dot\cup}
      \bigl(Y_{-,n}\cap\delta_{n-1}^-\bigr),
      \label{eq:whole-cylinder-comparison-minus}\\
  Y_{+,n}
    &=\bigl(Y_{+,n}\cap\delta_n^+\bigr)
      \mathbin{\dot\cup}
      \bigl(Y_{+,n}\cap\delta_{n-1}^+\bigr),
      \label{eq:whole-cylinder-comparison-plus}
\end{align}
modulo endpoints.  Thus every CMT return atom, not merely a subfamily of
them, is covered by one of the CLM half-return cases, with the index shift
made explicit in both directions.

It remains to verify \textup{(F2)(2)} on precisely the atoms for which that
condition is required.  Fix a sign with $k_\pm\ne1$ and an atom
$Y_{\pm,n}\not\subset U_{0\pm}$.  Since
$Y_{\pm,n}\subset Y_\pm\subset\Delta_0^\pm$ modulo endpoints, the definition
of $n_\pm$ shows that $n_\pm\ge1$: if $n_\pm=0$, then
$\Delta_0^\pm\subset U_{0\pm}$ and every such atom is contained in
$U_{0\pm}$.  Moreover, the intervals $\delta_r^\pm$ are ordered towards
$0$, so $\delta_r^\pm\subset U_{0\pm}$ for every $r\ge n_\pm$.  By the
decompositions above, $Y_{\pm,n}\not\subset U_{0\pm}$ therefore implies
$n\le n_\pm$.

Consider first $Y_{-,n}$.  On the component
$Y_{-,n}\cap\delta_n^-$, condition \textup{(A2)} gives
$(g^n)'(x)>\lambda_A$.  On $Y_{-,n}\cap\delta_{n-1}^-$, put
$z=g^{n-1}x$.  Then $z\in\Delta_0^+$ and
$gz=g^nx\in Y_-\subset\Delta_0^-$, so
$z\in\delta_1^+$ and consequently
\[
  x\in\delta_{n-1,1}^-.
\]
Thus $g^n$ on this component is a branch of the full first-return map $G_-$,
and \eqref{eq:full-return-expansion} gives $(g^n)'(x)>\lambda_R$.  The two
components of $Y_{+,n}$ are treated identically with the signs interchanged.
Taking $\lambda_F=\min\{\lambda_A,\lambda_R\}>1$ gives
\eqref{eq:F2-expansion}.  Finally, the formulas for
$\alpha_\pm$ follow directly from \eqref{eq:stickiness}.
\end{proof}

\begin{remark}
\label{rem:endpoint-conventions}
The following tables fix the signs by the dynamics rather than by typography
in the cited formulas:
\[
\begin{array}{c|c|c|c|c}
\text{CMT return atom}&\text{initial branch}&\text{endpoint visited}
 &\text{crossing coefficient}&\text{endpoint data}\\ \hline
Y_{-,n}&g_-&+1&a_-&(\ell_+,b_+)\\
Y_{+,n}&g_+&-1&a_+&(\ell_-,b_-)
\end{array}
\]
For the balanced exponent $\alpha$, set
\[
  B_+=a_-^{-1/k_-}(\ell_+b_+)^{-\alpha},
  \qquad
  B_-=a_+^{-1/k_+}(\ell_-b_-)^{-\alpha}.
\]
All set identities in the following table are understood modulo partition
endpoints.  Then, for $n\ge1$,
\[
\begin{array}{c|c|c|c}
\text{endpoint }\sigma&\text{occupation level}&\text{CLM tail constant}
 &\text{CMT tail constant}\\ \hline
+&\{\tau^{(+)}=n\}=Y_{-,n+1}&C_+=h(0-)B_+&c_+=h(0)B_+\\
-&\{\tau^{(-)}=n\}=Y_{+,n+1}&C_-=h(0+)B_-&c_-=h(0)B_-
\end{array}
\]
Here $h$ is the density of the same normalized invariant measure with respect
to Lebesgue; continuity at $0$ will imply that the two columns agree.
\end{remark}

\begin{theorem}
\label{thm:Aaronson-criterion}
Let $(Z,\mathcal B,\rho,T,\mathcal Q)$ be a conservative Markov fibred
system on a standard probability space: $\mathcal Q$ is a countable
generating Markov partition, and $T$ is nonsingular and invertible on each
atom of $\mathcal Q$.  Suppose that the transition matrix is aperiodic and
that the branch Radon--Nikodym Jacobians have the uniform R\'enyi property
\begin{equation}
 \sup_{N\ge1}\
 \sup_{\substack{C\in\mathcal Q^{(N)}\\ \rho(C)>0}}
 \operatorname*{ess\,osc}_{C}\log J_\rho T^N<\infty,
 \qquad
 \mathcal Q^{(N)}=\bigvee_{j=0}^{N-1}T^{-j}\mathcal Q.
 \label{eq:Renyi-criterion}
\end{equation}
Here, on every measurable set $E$ contained in a branch on which $T$ is
injective, $J_\rho T$ is determined by
\[
  \rho(TE)=\int_EJ_\rho T\,d\rho.
\]
The cited sources formulate the R\'enyi condition using inverse-branch
Radon--Nikodym derivatives.  If $v$ is the inverse of an injective branch,
then $J_\rho v=(J_\rho T\circ v)^{-1}$ almost everywhere, so the logarithms
have the same essential oscillation under the branch correspondence.
Then $T$ is exact.
\end{theorem}

\begin{proof}
Let $D$ denote the supremum in \eqref{eq:Renyi-criterion}.  On every
positive-measure cylinder $C\in\mathcal Q^{(N)}$,
\[
  e^{-D}\le
  \frac{J_\rho T^N(x)}{J_\rho T^N(y)}
  \le e^D
  \qquad\text{for }\rho\times\rho\text{-almost every }(x,y)\in C\times C.
\]
Thus every positive-measure cylinder has bounded metric distortion, with the
same constant $e^D$, in the sense of \cite[Section~2]{ADU93}.  Take as the
Schweiger collection all positive-measure cylinders.  It covers $Z$ modulo
$\rho$ (already the one-cylinders do), and if a cylinder belongs to this
collection, then prepending any admissible partition atom produces another
positive-measure cylinder with the same uniform distortion bound.  Therefore
the system has the Schweiger property.
Exactness now follows from \cite[Theorem~3.2]{ADU93}, which states that a
conservative, aperiodic Markov fibred system with the Schweiger property is
exact.  This is also the R\'enyi special case of
\cite[Theorem~4.4.7]{Aar97}.
\end{proof}

\begin{lemma}
\label{lem:tower-exactness}
For every $g\in\mathfrak F_*$, let
$G:\Delta_0^-\to\Delta_0^-$ be the first-return map of \cite{CLM23}, with
invariant probability $\widehat\mu$.  Its Kakutani tower, equipped with the
standard invariant sigma-finite measure, admits an equivalent probability
reference for which it is a conservative, irreducible and aperiodic Markov
fibred system with a generating partition and the R\'enyi property.  Hence
the tower is exact.  The factor map onto the interval then shows that $g$ is
conservative and exact in the invariant measure class, and therefore also in
the equivalent Lebesgue measure class.
\end{lemma}

\begin{proof}
Write $\tau_0$ for the return time and set
\[
  \Delta=\{(x,\ell):x\in\Delta_0^-,\ 0\le\ell<\tau_0(x)\},
\]
with tower map
\[
  T_\Delta(x,\ell)=
  \begin{cases}
    (x,\ell+1),&\ell+1<\tau_0(x),\\
    (Gx,0),&\ell+1=\tau_0(x).
  \end{cases}
\]
The return partition is
$\mathcal P^- =\{\delta_{m,n}^-:m,n\ge1\}$, and
\cite[equations~(38)--(40)]{CLM23} gives
\begin{equation}
  \tau_0=m+n,
  \qquad
  G|_{\delta_{m,n}^-}=g^{m+n}:\delta_{m,n}^-\to\Delta_0^-
  \quad\text{bijectively}.
  \label{eq:two-index-return-branches}
\end{equation}
Equip the tower with the standard lift
\begin{equation}
  \overline\mu
  =\sum_{\ell\ge0}
    \bigl(\widehat\mu|_{\{\tau_0>\ell\}}\bigr)\otimes\delta_\ell
  \label{eq:standard-tower-measure}
\end{equation}
and use the countable partition
\[
  \mathcal Q=
  \{\delta_{m,n}^-\times\{\ell\}:m,n\ge1,\ 0\le\ell<m+n\}.
\]
Every level has finite measure, so $\overline\mu$ is sigma-finite, including
when $\overline\mu(\Delta)=\infty$.

\emph{Invariance, nonsingularity, and conservativity.}
For every nonnegative measurable $\Phi$, separation of the vertical and top
levels gives
\begin{align*}
 \int_\Delta \Phi\circ T_\Delta\,d\overline\mu
 &=\sum_{\ell\ge0}\int_{\{\tau_0>\ell+1\}}
      \Phi(x,\ell+1)\,d\widehat\mu(x)
   +\sum_{\omega\in\mathcal P^-}\int_\omega
      \Phi(Gx,0)\,d\widehat\mu(x)\\
 &=\sum_{j\ge1}\int_{\{\tau_0>j\}}\Phi(x,j)\,d\widehat\mu(x)
   +\int_{\Delta_0^-}\Phi(y,0)\,d\widehat\mu(y)\\
 &=\int_\Delta\Phi\,d\overline\mu,
\end{align*}
because $G$ preserves $\widehat\mu$.  Hence $T_\Delta$ preserves
$\overline\mu$ and is nonsingular.  Since every tower point reaches the base
in finitely many iterates and the probability-preserving base map $G$ is
conservative, the standard return argument shows that $T_\Delta$ is
conservative.

\emph{Markov and generating properties.}
A vertical atom is mapped bijectively and nonsingularly onto the next atom in
its column, while the top atom of every column is mapped bijectively and
nonsingularly onto the whole base, a union of level-zero atoms.  Thus
$\mathcal Q$ has the Markov and local-invertibility properties.  It is
generating modulo the union of partition boundaries.  Indeed, two points
with the same $\mathcal Q$-itinerary have the same level at every time and,
at successive returns, their base coordinates have the same
$\mathcal P^-$-itinerary.  Uniform expansion of the Gibbs--Markov map $G$
forces the diameters of its cylinders to tend to zero, so the base
coordinates, and then the tower points, coincide.  Since
$\widehat\mu\sim\Leb$ on the base, the omitted boundaries are
$\overline\mu$-null.

\emph{The invariant-measure Jacobian and the R\'enyi property.}
The Jacobian in Theorem~\ref{thm:Aaronson-criterion} is computed with
respect to $\overline\mu$.  On every vertical atom,
\begin{equation}
  J_{\overline\mu}T_\Delta(x,\ell)=1
  \qquad(\ell+1<\tau_0(x)).
  \label{eq:vertical-invariant-Jacobian}
\end{equation}
If $\omega=\delta_{m,n}^-$, $q=m+n$, and
$d\widehat\mu=\widehat h\,d\Leb$, then on the top of that column
\begin{equation}
  J_{\overline\mu}T_\Delta(x,q-1)
  =J_{\widehat\mu}G(x)
  =\frac{\widehat h(Gx)|G'(x)|}{\widehat h(x)}.
  \label{eq:top-invariant-Jacobian}
\end{equation}
This follows directly from the defining identity for branch Jacobians:
for $E\subset\omega$,
\[
 \overline\mu\bigl(T_\Delta(E\times\{q-1\})\bigr)
 =\widehat\mu(GE)
 =\int_EJ_{\widehat\mu}G\,d\widehat\mu.
\]
Thus no Euclidean derivative is being substituted for the required
Radon--Nikodym Jacobian.

The Gibbs--Markov distortion estimate
\cite[Corollary~3.11]{CLM23}, together with the invariant density supplied in
the proof of \cite[Theorem~A]{CLM23}, yields constants $C_0>0$ and
$\theta\in(0,1)$ such
that, for $u,v$ in one return atom,
\begin{equation}
  \bigl|\log J_{\widehat\mu}G(u)
        -\log J_{\widehat\mu}G(v)\bigr|
  \le C_0\theta^{s_G(u,v)},
  \label{eq:base-Jacobian-distortion}
\end{equation}
where $s_G$ is separation time for $\mathcal P^-$.  Indeed, $\widehat h$ is
Lipschitz and bounded above and away from zero on the base, so
$\log\widehat h$ is Lipschitz; uniform expansion converts its Euclidean
variation on common return cylinders into a symbolic-Lipschitz bound.
Together with Corollary~3.11, this applies term by term to
\[
  \log J_{\widehat\mu}G
  =\log|G'|+\log\widehat h\circ G-\log\widehat h,
\]
and gives \eqref{eq:base-Jacobian-distortion}.

Let $C$ be a positive-measure $N$-cylinder for $\mathcal Q$, and suppose that its
itinerary contains $r$ completed returns.  The vertical factors in
\eqref{eq:vertical-invariant-Jacobian} contribute no oscillation.  At the
$j$th completed return, $0\le j<r$, the two pre-return base points have the
same current return atom and the same remaining $r-j-1$ return symbols;
hence their separation time is at least $r-j$.  Summing
\eqref{eq:base-Jacobian-distortion} over the completed returns gives
\[
  \operatorname*{ess\,osc}_{C}
       \log J_{\overline\mu}T_\Delta^N
  \le C_0\sum_{j=0}^{r-1}\theta^{r-j}
  \le \frac{C_0\theta}{1-\theta}.
\]
For $r=0$ the oscillation is zero.  This proves
\eqref{eq:Renyi-criterion} uniformly on all tower cylinders, including
cylinders crossing arbitrarily many returns.

The exactness theorem is formulated with a probability reference, whereas
$\overline\mu$ can have infinite mass.  This causes no difficulty.  Define
\[
  d\rho_\Delta(x,\ell)=Z^{-1}2^{-\ell}\,d\overline\mu(x,\ell),
  \qquad
  Z=\sum_{\ell\ge0}2^{-\ell}
       \widehat\mu(\tau_0>\ell)\in(0,2].
\]
Then $\rho_\Delta$ is a probability equivalent to $\overline\mu$.  If $C$ is
an $N$-cylinder, the levels $\ell_0$ and $\ell_N$ at times $0$ and $N$ are
constant on $C$.  The change-of-reference formula for branch Jacobians gives
\begin{equation}
  J_{\rho_\Delta}T_\Delta^N(x)
  =2^{-(\ell_N-\ell_0)}J_{\overline\mu}T_\Delta^N(x),
  \qquad x\in C.
  \label{eq:tower-probability-Jacobian}
\end{equation}
The prefactor is constant on $C$, so the essential oscillation of its
logarithm is zero.  Consequently the R\'enyi bound just proved for
$\overline\mu$ holds with exactly the same constant for the probability
reference $\rho_\Delta$.  Conservativity, null sets, the Markov property, and
the generating property are unchanged on passing to an equivalent measure.

\emph{Irreducibility and aperiodicity.}
For $Q=\delta_{m,n}^-\times\{\ell\}$ put
$r(Q)=m+n-\ell$.  Then $T_\Delta^{r(Q)}|_Q$ maps $Q$ bijectively onto the
base.  Fix another atom
$Q'=\delta_{m',n'}^-\times\{\ell'\}$.  Whenever
$N\ge r(Q)+\ell'+2$, set
$q=N-r(Q)-\ell'\ge2$.  Every integer $q\ge2$ occurs as a return time, for
example on $\omega_q=\delta_{1,q-1}^-$.  Since
$G|_{\omega_q}$ maps $\omega_q$ bijectively onto the base,
\[
  (G|_{\omega_q})^{-1}(\delta_{m',n'}^-)
\]
is a nondegenerate interval and has positive $\widehat\mu$-measure.  Pulling
this set back through $T_\Delta^{r(Q)}|_Q$ produces a positive-measure subset
of $Q$ whose image, after the chosen return and $\ell'$ vertical steps, lies
in $Q'$.  Hence
\[
  \rho_\Delta\bigl(Q\cap T_\Delta^{-N}Q'\bigr)>0
  \qquad\text{for all }N\ge r(Q)+\ell'+2.
\]
Thus the transition matrix is mixing, and in particular irreducible and
aperiodic.

All hypotheses of Theorem~\ref{thm:Aaronson-criterion} have now been
verified for $(\Delta,\rho_\Delta,T_\Delta,\mathcal Q)$.  Therefore
$T_\Delta$ is exact modulo $\rho_\Delta$, equivalently modulo
$\overline\mu$.

\emph{Passage to the interval factor.}
Define
\[
  \pi:\Delta\to I,
  \qquad
  \pi(x,\ell)=g^\ell x.
\]
Outside the countable union of partition-boundary orbits,
$\pi\circ T_\Delta=g\circ\pi$; hence this identity holds modulo
$\overline\mu$.  Moreover,
\begin{equation}
  \pi_*\overline\mu
  =\sum_{\ell\ge0}g_*^\ell
       \bigl(\widehat\mu|_{\{\tau_0>\ell\}}\bigr)
  =:\widetilde\mu.
  \label{eq:correct-first-return-lift}
\end{equation}
The strict inequality $\{\tau_0>\ell\}$ is forced by the tower definition:
a point contributes to level $\ell$ exactly when $\ell<\tau_0$.  The display
\cite[equation~(21)]{CLM23} writes $\{\tau_0\ge n\}$ next to $g_*^n$; with
the same level index this is an off-by-one typographical discrepancy.  We
use the invariant lift \eqref{eq:correct-first-return-lift}.

By \cite[Lemma~3.2 and Theorem~A]{CLM23}, the tower saturates $I$ modulo
Lebesgue measure zero and $\widetilde\mu\sim\Leb$.  Hence $\pi$ is a
surjective measure-preserving factor modulo null sets.  Conservativity
descends immediately.  To see that exactness descends, let
$B\in\bigcap_{N\ge0}g^{-N}\mathcal B_I$.  For each $N$, choose $B_N$ with
$B=g^{-N}B_N$ modulo $\widetilde\mu$.  Then
\[
  \pi^{-1}B=T_\Delta^{-N}(\pi^{-1}B_N)
  \quad\text{modulo }\overline\mu
\]
for every $N$.  Thus $\pi^{-1}B$ belongs to the tail sigma-algebra of
$T_\Delta$ and is null or conull.  Since
$\pi_*\overline\mu=\widetilde\mu$, the same is true of $B$.  Therefore $g$
is conservative and exact for $\widetilde\mu$.  Since
$\widetilde\mu\sim\Leb\sim m$, these are also properties of the Lebesgue
measure class.  In particular, $g$ satisfies \textup{(S1)}.
\end{proof}

\subsection{Verification for the balanced class}

The application theorem in \cite[Section~4.2]{CMT26} is stated for the common excursion exponent $\alpha\in(0,1)$.  Since Theorem~C of that paper allows $\alpha=1$, we first verify explicitly that its inducing construction and hypotheses continue to hold at the endpoint exponent.

\begin{lemma}
\label{lem:alpha-one-CMT}
Let $g\in\mathfrak F_*$ and suppose that $\beta=1$, equivalently $\alpha=1$.
Then, with $X=I$, $f=g$, $d=2$, $\xi_1=-1$, and $\xi_2=1$, the inducing
scheme of \cite[Section~4.2]{CMT26} satisfies \textup{(S1)--(S4)} and
hypotheses \textup{(H1)--(H4)} stated above, with exponent $\alpha=1$.
\end{lemma}

\begin{proof}
By Lemma~\ref{lem:class-compatibility}, $g\in\mathcal F_{\mathrm{CMT}}$.  The
endpoint labels are those in Remark~\ref{rem:endpoint-conventions}: from
\cite[equation~(4.9)]{CMT26},
\[
  g:Y_{-,n}\to X_{+,n-1},
  \qquad
  g:Y_{+,n}\to X_{-,n-1},
\]
so $Y_{-,n}$ visits $+1$ and $Y_{+,n}$ visits $-1$.  Balance and $\beta=1$
give $\alpha_+=\alpha_-=1$.

\emph{Standing assumptions.}
Let $F=g^\tau:Y\to Y$ be the first-return map.  The expansion proof in
\cite[Lemma~4.13]{CMT26} uses only \textup{(F0)--(F2)} and positivity of the
local exponents.  Its comparison-map induction therefore gives a constant
$\Lambda_F>1$, independent of the return atom, such that
\begin{equation}
  |F'(x)|\ge\Lambda_F
  \qquad(x\in Y_{\epsilon,n},\ \epsilon\in\{+,-\},\ n\ge2).
  \label{eq:alpha-one-uniform-expansion}
\end{equation}
For completeness, we isolate the only estimates in the distortion argument
where an endpoint exponent occurs.  In the notation of
\cite[equation~(4.11)]{CMT26}, the logarithmic distortion on
$Y_{\epsilon,n}$ is controlled by a first-passage term and a neutral-leg
sum.  Lemma~4.11 of that paper and condition \textup{(F2)} give, uniformly in
$n$ and in the intermediate points $u_k$ used there,
\begin{equation}
 \frac{|g''(u_0)|}{|g'(u_0)|}\,|Y_{\epsilon,n}|
       \le Cn^{-1},
 \qquad
 \frac{|g''(u_k)|}{|g'(u_k)|}\,|X_{-\epsilon,n-k}|
       \le C(n-k)^{-2}\quad(1\le k<n).
 \label{eq:alpha-one-distortion-terms}
\end{equation}
Here $-\epsilon$ denotes the sign opposite to $\epsilon$.
Indeed, the first product is bounded by
$Cn^{\alpha_\epsilon}n^{-1-\alpha_\epsilon}=Cn^{-1}$, while the second is
bounded by
\[
 C(n-k)^{-1+1/\ell_{-\epsilon}}
   (n-k)^{-1-1/\ell_{-\epsilon}}
 =C(n-k)^{-2}.
\]
These calculations require only positive local exponents and remain valid
when $\alpha_\epsilon=1$.  Since
$\sum_{j\ge1}j^{-2}<\infty$, the proof of
\cite[Lemma~4.15]{CMT26} now yields constants $C_D<\infty$ and
$\theta\in(0,1)$, independent of the atom, such that
\begin{equation}
  \left|\log\frac{|F'(x)|}{|F'(y)|}\right|
  \le C_D|Fx-Fy|
  \le C_D\theta^{s_F(x,y)}
  \qquad(x,y\in Y_{\epsilon,n}).
  \label{eq:alpha-one-GM-distortion}
\end{equation}
Every $F:Y_{\epsilon,n}\to Y$ is bijective.  Equations
\eqref{eq:alpha-one-uniform-expansion} and
\eqref{eq:alpha-one-GM-distortion} therefore prove directly that $F$ is a
full-branch Gibbs--Markov map; full branching implies topological mixing.
This proves at $\alpha=1$ the conclusion of
\cite[Proposition~4.12]{CMT26}, and hence \textup{(S2)} and the assertion
about $F$ in \textup{(S3)}.

The partition $\{Y_{\pm,j},X_{\pm,n}:j\ge2,n\ge1\}$ is Markov for $g$ by
\eqref{eq:CMT-partition}.  It is also generating modulo $m$.  Let
$\mathcal E$ be the union of all partition endpoints and all their backward
iterates.  This set is countable, hence null for both $m$ and the equivalent
measure $\widetilde\mu$.  Suppose that $x,x'\in I\setminus\mathcal E$ have
the same partition itinerary and initially lie in the atom $A$.  There is an
integer $r_A\ge1$ for which $g^{r_A}|_A$ maps $A$ bijectively onto $Y$.
The points $y=g^{r_A}x$ and $y'=g^{r_A}x'$ then have the same successive
return atoms for $F$.  If $\Lambda>1$ is the uniform expansion constant of
$F$, then for every $q\ge1$ the points $y,y'$ lie in a common $q$-cylinder,
and therefore
\[
  |y-y'|\le \operatorname{diam}(Y)\Lambda^{-q}.
\]
Letting $q\to\infty$ gives $y=y'$, and injectivity of $g^{r_A}|_A$ gives
$x=x'$.  Thus the separation condition in the abstract Markov setup holds
modulo both $m$ and $\widetilde\mu$.  Lemma~\ref{lem:tower-exactness}
gives conservativity and exactness, and exactness implies ergodicity.
Together with the generating Markov partition above, this verifies
\textup{(S1)}.  To check the remaining part of
\textup{(S3)}, fix source and target atoms $A,B$.  There is $r_A\ge1$ such
that $g^{r_A}|_A$ maps $A$ bijectively onto $Y$.  Choose $C_B\subset Y$ and
$s_B\in\{0,1\}$ so that $g^{s_B}|_{C_B}$ maps $C_B$ bijectively onto $B$:
use $C_B=B,s_B=0$ for a return atom and
$C_B=Y_{\mp,j+1},s_B=1$ for $B=X_{\pm,j}$.  For
$n\ge r_A+s_B+2$, put $q=n-r_A-s_B$.  A return branch $Y_{\pm,q}$ maps
bijectively onto $Y$ under $g^q$; restricting it to the inverse image of
$C_B$ and pulling back through $g^{r_A}|_A$ gives a nonempty subset of $A$
mapped into $B$ by $g^n$.  Thus $g^nA\cap B\ne\varnothing$ for all large
$n$, proving topological mixing.

The lifted measure $\widetilde\mu$ in
\eqref{eq:correct-first-return-lift} is the unique, up to scaling,
invariant absolutely continuous sigma-finite measure supplied by
\cite[Theorem~A]{CLM23}; it is equivalent to $m$.  Also
$0<\widetilde\mu(Y)<\infty$ by \cite[Lemma~2.1]{CMT26}, as required by the
inducing framework.  Infinitude of the total mass, the remaining point in
\textup{(S4)}, follows below from the return tail.

\emph{Hypotheses \textup{(H1)} and \textup{(H2)(a)--(b)}.}
The intervals $X_{-,n}$ and $X_{+,n}$ accumulate only at $-1$ and $1$,
respectively.  Given $\eps>0$, all but finitely many lie in
$U_\eps^-$ or $U_\eps^+$.  Hence
$I\setminus(U_\eps^-\cup U_\eps^+)$ is contained, modulo endpoints, in $Y$
and finitely many $X_{\pm,n}$.  Since $X_{\pm,n}\subset g^{-n}Y$,
invariance gives
\[
  \widetilde\mu(X_{\pm,n})\le\widetilde\mu(g^{-n}Y)
  =\widetilde\mu(Y)<\infty,
\]
which proves \textup{(H1)}.  Put
\[
  X_-=[-1,\gamma_-),
  \qquad
  X_+=(\gamma_+,1].
\]
Lemma~4.10 of \cite{CMT26} states that $Y$ dynamically separates these two
sets, proving \textup{(H2)(a)--(b)}.  In this concrete construction
$g(Y)=I\setminus Y$ modulo endpoints, so $\tau\ge2$ almost everywhere and
the return-time-one set is empty.  More generally, the zero-occupation case
in \textup{(H2)(b)} must be treated separately from the union formula for
$n\ge2$, as was done in the abstract statement above.

\emph{Lipschitz continuity and normalization of the induced density.}
Write
\[
  h_{\Leb}=\frac{d\widetilde\mu}{d\Leb}\bigg|_Y.
\]
We reproduce the transfer-operator estimate from the last part of the proof
of \cite[Theorem~4.9]{CMT26} to make clear that it is uniform at
$\alpha=1$.  Let $\mathcal L_Y$ be the transfer operator of $F$ with respect
to Lebesgue.  If $a$ is an $n$-cylinder and $v_a:Y\to a$ is the inverse
branch of $F^n$, uniform expansion and the distortion estimate
\cite[equation~(4.12)]{CMT26} give constants independent of $a,n$ such that
\[
  |v_a'(x)|\le C|a|,
  \qquad
  \bigl||v_a'(x)|-|v_a'(y)|\bigr|\le C|a|\,|x-y|.
\]
Since the $n$-cylinders partition $Y$ modulo endpoints,
\[
  \mathcal L_Y^n1(x)=\sum_a|v_a'(x)|
\]
is uniformly bounded and uniformly Lipschitz in $n$.  The Ces\`aro averages
$H_n=n^{-1}\sum_{j=0}^{n-1}\mathcal L_Y^j1$ therefore have a uniformly
convergent subsequence by Arzel\`a--Ascoli.  Moreover,
\[
  \mathcal L_YH_n-H_n=\frac{\mathcal L_Y^n1-1}{n}\longrightarrow0
  \quad\text{uniformly},
\]
The branch formula and the uniform bound on $\mathcal L_Y1$ imply
$\|\mathcal L_Yv\|_\infty\le C\|v\|_\infty$.  Hence $\mathcal L_Y$ is
continuous for uniform convergence, and passing to the chosen subsequence in
the last display gives $\mathcal L_Yh_0=h_0$.  Thus $h_0$ is an invariant
Lipschitz density.  Since
$\mathcal L_Y$ preserves integrals,
$\int_Yh_0\,d\Leb=\Leb(Y)$, and hence $h_0/\Leb(Y)$ is an invariant
probability density.  Uniqueness of the Gibbs--Markov absolutely continuous
invariant probability gives
\[
  \frac{h_0}{\Leb(Y)}
  =\frac{h_{\Leb}}{\widetilde\mu(Y)}
  \quad\text{almost everywhere on }Y.
\]
Consequently
$h_{\Leb}=(\widetilde\mu(Y)/\Leb(Y))h_0$ has a Lipschitz representative on
$Y$.  Positivity and
boundedness away from zero follow from \cite[Lemma~2.1]{CMT26}, so
\[
  h_{\Leb}(0-)=h_{\Leb}(0+)=:h_{\Leb}(0)>0.
\]
This argument uses only expansion and distortion and not the tail exponent.
Since $m=\Leb/2$, the density relative to $m$ is
$h_m=2h_{\Leb}$.  For every measurable $A\subset Y$,
\[
  \int_Ah_{\Leb}\,d\Leb=\int_Ah_m\,dm,
\]
so the change of reference measure does not change any tail constant.

\emph{Hypothesis \textup{(H2)(c)} and infinitude.}
At $\alpha=1$, Remark~\ref{rem:endpoint-conventions} gives
\[
  B_+=a_-^{-1/k_-}(\ell_+b_+)^{-1},
  \qquad
  B_-=a_+^{-1/k_+}(\ell_-b_-)^{-1}.
\]
We use Lemma~4.11 of \cite{CMT26}.  In endpoint-indexed notation it gives
\[
  |Y_{-,n}|\sim B_+n^{-2},
  \qquad
  |Y_{+,n}|\sim B_-n^{-2}.
\]
The displayed constants immediately preceding Theorem~4.9 in that paper are
not used here; the argument below relies directly on Lemma~4.11 and tail
summation.

Let $\tau^{(+)}$ and $\tau^{(-)}$ be the occupation variables for $X_+$ and
$X_-$.  For every $n\ge1$, modulo the endpoints of the Markov partition,
\[
  \{\tau^{(+)}=n\}=Y_{-,n+1},
  \qquad
  \{\tau^{(-)}=n\}=Y_{+,n+1}.
\]
Consequently
\[
  \Leb(\tau^{(\sigma)}=n)\sim B_\sigma n^{-2},
  \qquad
  \widetilde\mu(\tau^{(\sigma)}=n)
  \sim h_{\Leb}(0)B_\sigma n^{-2}.
\]
Summing the levels and using $\sum_{j>n}j^{-2}\sim n^{-1}$ gives
\begin{equation}
  \widetilde\mu(\tau^{(\sigma)}>n)
  \sim h_{\Leb}(0)B_\sigma n^{-1}.
  \label{eq:alpha-one-tail}
\end{equation}
Thus \textup{(H2)(c)} holds with
$c_+=h_{\Leb}(0)B_+$ and $c_-=h_{\Leb}(0)B_-$.  Since
$\tau=1+\tau^{(+)}+\tau^{(-)}$,
\[
  \int_Y\tau\,d\widetilde\mu
  =\sum_{n\ge0}\widetilde\mu(\tau>n)=\infty.
\]
The return density and its reciprocal are bounded on $Y$, so the same
nonintegrability holds with respect to Lebesgue on $Y$.  The final assertion
of \cite[Lemma~2.1]{CMT26} gives $\widetilde\mu(I)=\infty$, completing
\textup{(S4)}.

\emph{Hypotheses \textup{(H3)} and \textup{(H4)}.}
Take $\mathcal K(I)=C^1(I)$, dense in $L^1(I,m)$.  The fixed-depth inverse
branch calculation in \cite[Lemma~4.16]{CMT26} gives, for every fixed
$r\ge1$ and $\rho\in C^1(I)$,
\[
  Q_r^\rho(y)
  =\frac{\rho(x_{+,r}(y))}{(g^r)'(x_{+,r}(y))}
   +\frac{\rho(x_{-,r}(y))}{(g^r)'(x_{-,r}(y))},
\]
where $x_{\pm,r}(y)\in X_{\pm,r}$ is the unique point satisfying
$g^r x_{\pm,r}(y)=y$.
There are also constants $C_r,C_r'$ such that
\[
  \|Q_r^\rho\|_\infty\le C_r\|\rho\|_{C^1},
  \qquad
  |Q_r^\rho(y)-Q_r^\rho(y')|
  \le C_r'\|\rho\|_{C^1}|y-y'|.
\]
The case $r=0$ is $Q_0^\rho=\rho|_Y$ and satisfies the same type of bounds.
Let $s_F$ be the separation time for the return partition, let $\theta_0<1$
be a Gibbs--Markov distortion parameter for $F$, and choose
\[
  \max\{\theta_0,\Lambda_F^{-1}\}\le\theta<1,
  \qquad \Lambda_F:=\inf_Y|F'|>1.
\]
(Increasing the symbolic parameter preserves the distortion bound.)  If
$y,y'$ lie in the same return atom and
$s_F(y,y')=s\ge1$, then
\[
  |y-y'|\le \operatorname{diam}(Y)\Lambda_F^{-s}
  \le \operatorname{diam}(Y)d_\theta(y,y').
\]
If they lie in different atoms, then $s_F(y,y')=0$ and $d_\theta(y,y')=1$,
while
$|Q_r^\rho(y)-Q_r^\rho(y')|\le2\|Q_r^\rho\|_\infty$.
Thus $Q_r^\rho\in\mathcal B_\theta(Y)$ for every $r$, proving
\textup{(H3)}.  Hypothesis \textup{(H4)} imposes an extra estimate only for
$\alpha\le\tfrac12$, and is therefore vacuous at $\alpha=1$.  All required
assumptions are verified.
\end{proof}

The constants in \eqref{eq:tails} come from the two-excursion return cycle used in \cite{CLM23,CL24}, whereas \cite{CMT26} uses the one-excursion first return to $Y$.  The following lemma makes the normalization and the endpoint labels explicit.

\begin{lemma}
\label{lem:tail-constant-identification}
Let $g\in\mathfrak F_*$ and let $c_+,c_->0$ be the constants in
\textup{(H2)(c)} for the inducing scheme of \cite[Section~4.2]{CMT26},
indexed by the endpoint visited.  Scale the unique invariant measure so that
its restriction to the left inducing base $\Delta_0^-$ is the probability
$\widehat\mu$ used in \eqref{eq:tails}.  Then
\[
  c_+=C_+,
  \qquad
  c_-=C_-.
\]
\end{lemma}

\begin{proof}
Set
\[
  B_+=a_-^{-1/k_-}(\ell_+b_+)^{-\alpha},
  \qquad
  B_-=a_+^{-1/k_+}(\ell_-b_-)^{-\alpha}.
\]
Let $\widetilde\mu$ have the normalization in the statement and write
$h_{\Leb}=d\widetilde\mu/d\Leb$ near $0$.  Corollary~2.4 and the explicit
constants in equations~(25)--(26) and Proposition~2.6 of \cite{CLM23}, as
used in \cite[Proposition~3.1]{CL24}, give
\begin{equation}
  C_+=h_{\Leb}(0-)B_+,
  \qquad
  C_-=h_{\Leb}(0+)B_-.
  \label{eq:CL-tail-formulas}
\end{equation}
The transfer-operator argument in the proof of \cite[Theorem~4.9]{CMT26},
and its exponent-one version in Lemma~\ref{lem:alpha-one-CMT}, show that
$h_{\Leb}$ is continuous on $Y$.  Hence both one-sided values equal
$h_{\Leb}(0)$.

For the CMT first return, modulo the same null set of partition endpoints,
the index relations are
\[
  \{\tau^{(+)}=n\}=Y_{-,n+1},
  \qquad
  \{\tau^{(-)}=n\}=Y_{+,n+1},
  \qquad n\ge1.
\]
Lemma~4.11 of \cite{CMT26}, with the endpoint indexing fixed in
Remark~\ref{rem:endpoint-conventions}, therefore gives
\[
  \Leb(\tau^{(\sigma)}=n)
  \sim \alpha B_\sigma n^{-1-\alpha}.
\]
The level intervals shrink to $0-$ for $\sigma=+$ and to $0+$ for
$\sigma=-$, so continuity of the density yields
\[
  \widetilde\mu(\tau^{(\sigma)}=n)
  \sim h_{\Leb}(0)\alpha B_\sigma n^{-1-\alpha}.
\]
Direct summation,
$\sum_{j>n}\alpha j^{-1-\alpha}\sim n^{-\alpha}$, gives
\[
  \widetilde\mu(\tau^{(\sigma)}>n)
  \sim h_{\Leb}(0)B_\sigma n^{-\alpha}.
\]
Thus $c_\sigma=h_{\Leb}(0)B_\sigma$, and comparison with
\eqref{eq:CL-tail-formulas} proves the claim.  If the density is instead
written relative to $m=\Leb/2$, it becomes $h_m=2h_{\Leb}$ while the level
measure becomes $dm=d\Leb/2$; their product, and hence the tail constants,
is unchanged.
\end{proof}


\begin{proposition}
\label{prop:SNM-balanced}
Every $g\in\mathfrak F_*$ has the strong natural-measure property \textup{(SNM)}.
\end{proposition}

\begin{proof}
For $g\in\mathfrak F_*$ we know that $\alpha=1/\beta\in(0,1]$.  Lemma~\ref{lem:tower-exactness} verifies
\textup{(S1)}, while Lemma~\ref{lem:class-compatibility} and
\cite[Proposition~4.12]{CMT26} verify \textup{(S2)}.  The Markov,
generating, and topological-mixing argument in the proof of
Lemma~\ref{lem:alpha-one-CMT} uses only \textup{(F0)--(F2)}, and hence
verifies \textup{(S3)} throughout the balanced class, not only when
$\alpha=1$.

For \textup{(S4)}, \cite[Theorem~A]{CLM23} supplies a unique, up to scaling,
invariant absolutely continuous sigma-finite measure $\widetilde\mu$
equivalent to Lebesgue.  Since $\beta\ge1$, \cite[Theorem~B]{CLM23} implies
that this measure cannot have finite total mass: otherwise its normalization
would be an invariant absolutely continuous probability measure, contrary to
that theorem and the uniqueness in Theorem~A.  Consequently
$\widetilde\mu(I)=\infty$, and \textup{(S4)} holds for every
$g\in\mathfrak F_*$.  Thus \textup{(S1)--(S4)} are verified on the whole
balanced class.

If $\alpha\in(0,1)$, Theorem~4.9 of \cite{CMT26} verifies
\textup{(H1)--(H4)}.  If $\alpha=1$, they are verified in
Lemma~\ref{lem:alpha-one-CMT}.  Theorem~\ref{thm:CMT-C} therefore gives, for
every probability $\lambda\ll m$,
\[
  (g^n)_*\lambda\weakstar
  \frac{c_+}{c_++c_-}\delta_1
  +\frac{c_-}{c_++c_-}\delta_{-1}.
\]
Lemma~\ref{lem:tail-constant-identification} identifies this measure with
$\nu_*$ from \eqref{eq:nustar}.  This proves the assertion.
\end{proof}

\begin{remark}
Proposition~\ref{prop:SNM-balanced} makes \textup{(SNM)} automatic in the present paper.  We nevertheless retain it as a named property in the statements below in order to separate the direct one-time input from the subsequent variance arguments.  The global--local theorem of \cite{CM25} has a different set of standing assumptions and is not needed for Proposition~\ref{prop:SNM-balanced}.
\end{remark}

The following formulation of \textup{(SNM)} is useful.

\begin{lemma}
\label{lem:mix-equivalence}
The strong natural-measure property \textup{(SNM)} is equivalent to
\begin{equation}
  \int_I \psi\,\phi\circ g^n\dd m
  \longrightarrow
  \left(\int_I\psi\dd m\right)
  \left(\int_I\phi\dd\nu_*\right)
  \label{eq:mix}
\end{equation}
for every $\psi\in L^1(m)$ and every $\phi\in C(I)$.
\end{lemma}

\begin{proof}
Assume the strong natural-measure property \textup{(SNM)}.  If $\psi\ge0$ and $\int\psi\dd m>0$, define the probability measure
\[
  \dd\lambda=\frac{\psi}{\int\psi\dd m}\dd m.
\]
Then \textup{(SNM)}, tested against $\phi$, gives the displayed correlation limit.  The case $\int\psi\dd m=0$ is trivial, and a general real-valued $\psi\in L^1(m)$ follows by applying the argument to its positive and negative parts.  Complex-valued functions follow by separating real and imaginary parts.

Conversely, suppose the displayed correlation limit holds.  If $\lambda\ll m$ is a probability with density $\psi=\dd\lambda/\dd m$, then
\[
  \int_I\phi\dd (g^n)_*\lambda
  =\int_I\psi\,\phi\circ g^n\dd m
  \longrightarrow\int_I\phi\dd\nu_*
\]
for every $\phi\in C(I)$, which is \textup{(SNM)}.
\end{proof}

For a probability $\lambda\ll m$, define the annealed sparse empirical measure
\[
  \overline\mu_{n,\lambda}^\mathbf{a}
  :=\int_I\mu_n^\mathbf{a}(x)\dd\lambda(x)
  =\frac1n\sum_{k=0}^{n-1}(g^{a_k})_*\lambda.
\]

\begin{maintheorem}
\label{thm:annealed}
Assume the strong natural-measure property \textup{(SNM)}.  Let $\mathbf{a}=(a_k)$ be any strictly increasing sequence of nonnegative integers and let $\lambda\ll m$ be a probability measure.  Then
\[
  \overline\mu_{n,\lambda}^\mathbf{a}\weakstar\nu_*.
\]
\end{maintheorem}

\begin{proof}
Fix $\phi\in C(I)$.  Since $a_k\to\infty$, \textup{(SNM)} gives
\[
  \int_I\phi\dd(g^{a_k})_*\lambda
  \longrightarrow
  \int_I\phi\dd\nu_*.
\]
Ces\`aro averaging therefore yields
\[
  \int_I\phi\dd\overline\mu_{n,\lambda}^\mathbf{a}
  =\frac1n\sum_{k=0}^{n-1}
    \int_I\phi\dd(g^{a_k})_*\lambda
  \longrightarrow
  \int_I\phi\dd\nu_*.
\]
This is weak-star convergence.
\end{proof}

\begin{remark}
Theorem~\ref{thm:annealed} is an averaged statement over initial conditions.  It does not imply that $\mu_n^\mathbf{a}(x)$ converges for almost every individual $x$.  A sufficient route to the latter conclusion is a fluctuation estimate, treated next.
\end{remark}

\section{Almost-sure sparse convergence}
\label{sec:strong-law}

\subsection{A variance criterion}

For bounded measurable random variables $F,H$ on $(I,m)$, write
\[
  \Cov_m(F,H)=\int_I FH\dd m-
  \left(\int_I F\dd m\right)
  \left(\int_I H\dd m\right).
\]

\begin{maintheorem}
\label{thm:variance}
Let $f:I\to I$ be measurable, let $\mathbf{a}=(a_k)$ be a strictly increasing sequence of nonnegative integers, and let $\nu$ be a Borel probability measure on $I$.  Let $\mathcal D\subset C(I)$ be countable and dense in the uniform norm.  Suppose that for every $\phi\in\mathcal D$ the following hold:
\begin{enumerate}[label=\textnormal{(V\arabic*)}]
\item
\begin{equation}
  \frac1n\sum_{k=0}^{n-1}\int_I\phi\circ f^{a_k}\dd m
  \longrightarrow\int_I\phi\dd\nu.
\end{equation}
\item there are constants $C_\phi<\infty$ and $\delta_\phi>0$ such that
\begin{equation}
  \int_I\left|
    \sum_{k=0}^{n-1}
    \left(\phi\circ f^{a_k}-\int_I\phi\circ f^{a_k}\dd m\right)
  \right|^2\dd m
  \le C_\phi n^{2-\delta_\phi}.
\end{equation}
for all $n\ge1$.
\end{enumerate}
Then, for $m$-almost every $x$,
\[
  \frac1n\sum_{k=0}^{n-1}\delta_{f^{a_k}x}\weakstar\nu.
\]
\end{maintheorem}

\begin{proof}
Fix $\phi\in\mathcal D$ and set
\[
  Y_k=\phi\circ f^{a_k}-\int_I\phi\circ f^{a_k}\dd m,
  \qquad
  S_n=\sum_{k=0}^{n-1}Y_k.
\]
Choose $r>\max\{1,1/\delta_\phi\}$ and put $n_q=\lfloor q^r\rfloor$.  Since $r>1$, this integer sequence is strictly increasing for all sufficiently large $q$.  Throughout the remainder of the argument we restrict to this cofinite range, which does not affect any limit.  By Chebyshev's inequality and \textup{(V2)}, for every $\eta>0$,
\[
  m\left(\left|S_{n_q}\right|>\eta n_q\right)
  \le \frac{C_\phi}{\eta^2 n_q^{\delta_\phi}}.
\]
Since $r\delta_\phi>1$, the right-hand side is summable in $q$.  Applying Borel--Cantelli for $\eta=1/\ell$, $\ell\in\mathbb N$, and intersecting the resulting full-measure sets gives
\[
  \frac{S_{n_q}(x)}{n_q}\longrightarrow0
\]
for almost every $x$.

The summands satisfy $|Y_k|\le2\|\phi\|_\infty$.  If $n_q\le n<n_{q+1}$, then
\[
\begin{aligned}
  \left|\frac{S_n}{n}-\frac{S_{n_q}}{n_q}\right|
  &\le \frac{|S_n-S_{n_q}|}{n}
     +|S_{n_q}|\left|\frac1n-\frac1{n_q}\right|\\
  &\le 4\|\phi\|_\infty
     \frac{n_{q+1}-n_q}{n_q}.
\end{aligned}
\]
Because $r>1$, $(n_{q+1}-n_q)/n_q\to0$.  Hence $S_n(x)/n\to0$ almost surely.  Combining this with \textup{(V1)} gives
\[
  \frac1n\sum_{k=0}^{n-1}\phi(f^{a_k}x)
  \longrightarrow\int_I\phi\dd\nu
\]
for almost every $x$.

Intersect the resulting full-measure sets over the countable family $\mathcal D$.  If $x$ belongs to this intersection and $\phi\in C(I)$ is arbitrary, choose $\psi\in\mathcal D$ with $\|\phi-\psi\|_\infty<\eta$.  Then
\[
  \limsup_{n\to\infty}
  \left|
    \frac1n\sum_{k=0}^{n-1}\phi(f^{a_k}x)-\int_I\phi\dd\nu
  \right|
  \le2\eta.
\]
Letting $\eta\downarrow0$ proves weak-star convergence.
\end{proof}

\begin{corollary}
\label{cor:covariance}
In Theorem~\ref{thm:variance}, condition \textup{(V2)} follows if, for every $\phi\in\mathcal D$, there are $C_\phi<\infty$ and $\delta_\phi>0$ such that
\begin{equation}
  \sum_{0\le i,j<n}
  \left|
    \Cov_m\bigl(\phi\circ f^{a_i},\phi\circ f^{a_j}\bigr)
  \right|
  \le C_\phi n^{2-\delta_\phi}.
  \label{eq:covariance-sum}
\end{equation}
for all $n\ge1$.
\end{corollary}

\begin{proof}
The variance in \textup{(V2)} is the sum of the covariances without absolute values.  Its absolute value is bounded by the left-hand side of \eqref{eq:covariance-sum}.
\end{proof}

\subsection{Existence of arbitrarily sparse statistical sampling}

The next theorem shows that the strong natural-measure property always permits a deterministic pointwise-convergent sampling sequence, although it does not identify a prescribed sequence such as $2^k$.

\begin{maintheorem}
\label{thm:selection}
Assume the strong natural-measure property \textup{(SNM)}, and let $(b_k)_{k\ge0}$ be any sequence of nonnegative integers.  Then there is a strictly increasing sequence $\mathbf{a}=(a_k)$ of nonnegative integers such that $a_k\ge b_k$ for every $k$ and
\[
  \mu_n^\mathbf{a}(x)\weakstar\nu_*
\]
for $m$-almost every $x$.
\end{maintheorem}

\begin{proof}
Choose a countable uniformly dense family
\[
  \mathcal D=\{\phi_0,\phi_1,\phi_2,\ldots\}\subset C(I).
\]
Set $a_{-1}=-1$ and construct $a_k$ inductively.  Suppose $a_0<\cdots<a_{k-1}$ have been chosen.  By Lemma~\ref{lem:mix-equivalence}, for each fixed $0\le r\le k$ and $0\le i<k$,
\[
  \Cov_m\bigl(\phi_r\circ g^{a_i},\phi_r\circ g^n\bigr)
  \longrightarrow0
  \qquad(n\to\infty).
\]
Indeed, apply \eqref{eq:mix} with
$\psi=\phi_r\circ g^{a_i}$ and $\phi=\phi_r$.  The first term in the
covariance converges to
$(\int_I\psi\dd m)(\int_I\phi_r\dd\nu_*)$, while \textup{(SNM)} applied to
$m$ gives
$\int_I\phi_r\circ g^n\dd m\to\int_I\phi_r\dd\nu_*$.  Subtracting the two
limits gives the asserted covariance convergence.  Also,
\[
  \int_I\phi_r\circ g^n\dd m
  \longrightarrow\int_I\phi_r\dd\nu_*.
\]
There are only finitely many conditions at stage $k$.  We may therefore choose
$a_k\in\N_0$ with
\[
  a_k>\max\{a_{k-1},b_k\}
\]
so large that
\begin{align}
  \left|
    \int_I\phi_r\circ g^{a_k}\dd m-
    \int_I\phi_r\dd\nu_*
  \right|&\le2^{-(k+1)}
    &&(0\le r\le k),\label{eq:selection-bias}\\
  \left|
    \Cov_m\bigl(\phi_r\circ g^{a_i},
                 \phi_r\circ g^{a_k}\bigr)
  \right|&\le2^{-(k+1)}
    &&(0\le r\le k,\ 0\le i<k).\label{eq:selection-cov}
\end{align}
For $k=0$, only the covariance family is empty.

Fix $r$.  Equation~\eqref{eq:selection-bias}, apart from finitely many initial terms, implies \textup{(V1)} with $\nu=\nu_*$.  For every pair whose later index is $j\ge r$, condition \eqref{eq:selection-cov} gives the bound $2^{-(j+1)}$; the finitely many pairs with later index $j<r$ are absorbed into a constant $C_r$.  The diagonal covariances are at most $\|\phi_r\|_\infty^2$, and therefore
\[
\begin{aligned}
  \sum_{0\le i,j<n}
  \left|\Cov_m(\phi_r\circ g^{a_i},\phi_r\circ g^{a_j})\right|
  &\le C_r+n\|\phi_r\|_\infty^2
    +2\sum_{j=r}^{n-1}j2^{-(j+1)}\\
  &\le C_r' n.
\end{aligned}
\]
Thus Corollary~\ref{cor:covariance} applies with $\delta_{\phi_r}=1$.  Theorem~\ref{thm:variance} yields the desired almost-sure weak-star convergence.
\end{proof}

\begin{remark}
The lower bounds $b_k$ are arbitrary.  For example, one may impose $a_k\ge2^{2^k}$, or any faster prescribed growth.  Thus sufficiently sparse deterministic observation can restore pointwise statistical convergence, even though every block-dominating sequence preserves non-convergence.
\end{remark}

\subsection{A dyadic criterion}

The following is an application of the abstract variance theorem.  It is not
an unconditional statement for every map in $\mathfrak F_*$; however,
Proposition~\ref{prop:g1-dyadic-example} in
Appendix~\ref{app:g1-dyadic} verifies its hypothesis for the explicit map
$g_1$.

\begin{corollary}
\label{cor:dyadic}
Assume the strong natural-measure property \textup{(SNM)}.  Let $\mathcal D\subset C(I)$ be countable and uniformly dense.  Suppose that for every $\phi\in\mathcal D$ there are constants $C_\phi<\infty$ and $\delta_\phi>0$ such that
\begin{equation}
  \sum_{0\le i,j<n}
  \left|
    \Cov_m\bigl(\phi\circ g^{2^i},
                 \phi\circ g^{2^j}\bigr)
  \right|
  \le C_\phi n^{2-\delta_\phi}.
  \label{eq:dyadic-covariance}
\end{equation}
for all $n\ge1$.  Then
\[
  \mu_n^{\mathrm{dyad}}(x)\weakstar\nu_*
\]
for $m$-almost every $x$, and hence $m(X_{\mathrm{dyad}})=0$.
\end{corollary}

\begin{proof}
The strong natural-measure property \textup{(SNM)} implies
\[
  \int_I\phi\circ g^{2^k}\dd m
  \longrightarrow\int_I\phi\dd\nu_*.
\]
Therefore \textup{(V1)} follows by Ces\`aro averaging for every $\phi\in C(I)$.  Condition~\eqref{eq:dyadic-covariance} implies \textup{(V2)} by Corollary~\ref{cor:covariance}.  Apply Theorem~\ref{thm:variance}.
\end{proof}

\begin{corollary}
\label{cor:gap}
In Corollary~\ref{cor:dyadic}, condition \eqref{eq:dyadic-covariance} holds with $\delta_\phi=1$ if, for every $\phi\in\mathcal D$, there is a summable sequence $r_\phi:\N_0\to[0,\infty)$ such that
\[
  \left|
    \Cov_m\bigl(\phi\circ g^{2^i},
                 \phi\circ g^{2^j}\bigr)
  \right|
  \le r_\phi(j-i)
  \qquad(0\le i\le j).
\]
\end{corollary}

\begin{proof}
We have
\[
  \sum_{0\le i,j<n}|\Cov_m(\cdot,\cdot)|
  \le nr_\phi(0)+2\sum_{h=1}^{n-1}(n-h)r_\phi(h)
  \le n\left(r_\phi(0)+2\sum_{h\ge1}r_\phi(h)\right).
\]
\end{proof}

\begin{remark}
\label{rem:dyadic-status}
The natural-measure theorem gives the one-time limits needed for \textup{(V1)}, but it does not imply \eqref{eq:dyadic-covariance}.  To obtain \eqref{eq:dyadic-covariance}, one needs a two-time estimate that is uniform as the earlier observation time varies.  A bounded-distortion statement for complete tower cylinders does not supply such a bound after conditioning on an arbitrary past event: a measurable cut of a cylinder can produce a pushforward density with a variation norm that is not controlled by its mass.  Consequently, the one-time operator-renewal and global--local mixing results cited above cannot be used as a substitute for \eqref{eq:dyadic-covariance} without an additional uniform argument.

Proposition~\ref{prop:g1-dyadic-example} supplies the additional uniform
complete-branch argument for $g_1$.  Verification of
\eqref{eq:dyadic-covariance} for every $g\in\mathfrak F_*$ is not established
here.
\end{remark}

\section{Conclusion}
\label{sec:conclusion}

The map-specific conclusions are now separated from the abstract sparse
strong-law machinery.  First, ordinary non-statistical behavior is robust
under every block-dominating deterministic observation scheme.  The proof
uses full-measure sets of genuinely dominating excursions on the two inducing
bases and a precise saturation/first-entrance transfer to the whole interval.
Positive-density, arithmetic, polynomial, and finite-index regularly varying
sampling all lie in this persistent regime.

Second, the strong natural-measure property holds for the whole balanced
class, including the exponent-one boundary.  The class inclusion is proved
by the explicit cylinder decompositions
\eqref{eq:cylinder-comparison-minus-a}--\eqref{eq:cylinder-comparison-plus-b}.
Exactness follows by applying the Markov-fibred-system criterion to an
explicit probability reference equivalent to the invariant sigma-finite
measure on the saturated tower, after verifying the branch Radon--Nikodym
Jacobian and uniform R\'enyi distortion; it then descends through the
measure-preserving factor.  The endpoint labels, density
normalizations, tower-level convention, and tail constants are reconciled
explicitly; the proof uses the level asymptotic in Lemma~4.11 of \cite{CMT26}
and direct summation rather than the conflicting displayed constants in that
source.  The symmetric polynomial family has natural measure
$\frac12(\delta_{-1}+\delta_1)$.

For $\beta>1$, the ordinary-time dimension theorem of Coates--Gelfert
\cite{CG26} complements the sparse-time conclusions above.  Their result and
Theorem~\ref{thm:CL} together give
\[
  m(\mathcal B_g(\nu_p))=0,
  \qquad
  \dim_{\mathrm H}\mathcal B_g(\nu_p)=1
  \qquad(p\in[0,1]).
\]
In particular, the distinguished mixture $\nu_*$ has a Lebesgue-null
ordinary-time basin but a basin of full Hausdorff dimension.  By contrast,
Theorem~\ref{thm:selection} chooses a single deterministic sampling sequence,
subject to arbitrary prescribed lower bounds on its growth, for which
\[
  \mu_n^\mathbf{a}(x)\weakstar\nu_*
  \qquad\text{for }m\text{-almost every }x.
\]
Thus endpoint mixtures can arise in two complementary ways: from exceptional
initial points under full-time observation, or from suitably chosen
observation schedules for Lebesgue-typical initial points.  The result of
\cite{CG26}, which assumes $\alpha\in(0,1)$, does not cover the boundary case
$\beta=1$ treated here.

The annealed theorem, variance criterion, covariance criterion, and diagonal
selection theorem are abstract consequences of one-time convergence and
second-moment estimates.  They imply annealed convergence along every
deterministic sequence and produce pointwise-convergent sampling sequences of
arbitrarily fast prescribed growth.  For an arbitrary map in
$\mathfrak F_*$, dyadic convergence remains conditional on the uniform
two-time estimate \eqref{eq:dyadic-covariance}.  In
Appendix~\ref{app:g1-dyadic}, Proposition~\ref{prop:g1-dyadic-example}
verifies this estimate on a countable uniformly dense endpoint-flat family
for the symmetric quadratic map $g_1$ and proves
almost-sure convergence to $\tfrac12(\delta_{-1}+\delta_1)$.

\appendix

\section{Example}
\label{ex:symmetric-family}
For every integer $p\ge1$, define $g_p:I\to I$ by
\[
  g_p(x)=
  \begin{cases}
    x+(1+x)^{p+1},&-1\le x\le0,\\[2mm]
    x-(1-x)^{p+1},&0<x\le1.
  \end{cases}
\]
The value $g_p(0)=1$ merely fixes a convention at the discontinuity.  On the two open branches,
\[
  g_p'(x)=
  \begin{cases}
    1+(p+1)(1+x)^p,&x<0,\\
    1+(p+1)(1-x)^p,&x>0,
  \end{cases}
\]
and
\[
  g_p''(x)=
  \begin{cases}
    p(p+1)(1+x)^{p-1}>0,&x<0,\\
    -p(p+1)(1-x)^{p-1}<0,&x>0.
  \end{cases}
\]
Hence both branches are full, $g_p'(x)>1$ away from the endpoints, and the left branch is convex while the right branch is concave.  Near the endpoints,
\[
  g_p(x)-x=(1+x)^{p+1},
  \qquad
  x-g_p(x)=(1-x)^{p+1},
\]
and $g_p'(0-)=g_p'(0+)=p+2>1$.  The finite-flight condition \textup{(A2)} also holds.  For fixed $n$, each one-sided closure $\overline{\delta_n^\pm}$ is compact: it is obtained by finitely many inverse images of compact interval closures under the continuous one-sided branch extensions.  The product derivative $(g_p^n)'$ extends continuously to this closure.  At every interior orbit point its factors are strictly larger than one.  If a boundary orbit lands on a neutral endpoint, all subsequent endpoint factors equal one, but at least one preceding factor---in particular the crossing factor when the discontinuity boundary is involved---is strictly larger than one; hence the full product remains strictly larger than one.  Thus $(g_p^n)'>1$ on the compact one-sided closure.  Only finitely many cylinders with $1\le n\le n_\pm$ occur, so their minima have a common lower bound $\lambda_A>1$; when $n_\pm=0$, the corresponding requirement is vacuous.  Thus, in the notation of \cite{CLM23,CL24},
\[
  \ell_1=\ell_2=p,
  \qquad
  k_1=k_2=1,
  \qquad
  \beta^- =\beta^+=p,
\]
so $g_p\in\mathfrak F_*$.
In particular, $g_1\in\mathfrak F_*$ and its balanced exponent is
$\beta=1$.

Under the affine change of coordinates $H(u)=2u-1$, the conjugate map $f_p=H^{-1}\circ g_p\circ H$ is
\[
  f_p(u)=
  \begin{cases}
    u+2^p u^{p+1},&0\le u\le\tfrac12,\\[2mm]
    u-2^p(1-u)^{p+1},&\tfrac12<u\le1.
  \end{cases}
\]
It is a two-neutral-fixed-point Thaler map with exponent $\alpha=1/p\in(0,1]$.  Hence \cite[Theorem~1.1]{CMT26} gives an endpoint-supported probability $\nu$ such that
\[
  (f_p^n)_*\lambda\weakstar\nu
\]
for every absolutely continuous probability measure $\lambda$ on $[0,1]$.  Let $R(u)=1-u$ and let $\lambda_0$ be normalized Lebesgue measure.  Away from the discontinuity, $f_p\circ R=R\circ f_p$; the discontinuity and its backward orbit are countable and therefore $\lambda_0$-null.  Since $\lambda_0$ is $R$-invariant, every $(f_p^n)_*\lambda_0$ is $R$-invariant, and so is its weak limit $\nu$.  The only $R$-invariant probability supported on $\{0,1\}$ is
\[
  \nu=\frac12\delta_0+\frac12\delta_1.
\]
The limit in \cite[Theorem~1.1]{CMT26} is the same for every absolutely continuous initial probability.  Conjugating by $H$ therefore gives
\[
  (g_p^n)_*\lambda\weakstar
  \frac12\delta_{-1}+\frac12\delta_1
  \qquad\text{for every probability }\lambda\ll m.
\]
By Proposition~\ref{prop:SNM-balanced} and Lemma~\ref{lem:tail-constant-identification}, this common limit is the measure $\nu_*$ defined from the Coates--Luzzatto tail constants.  Hence $p_*=1/2$, and \eqref{eq:nustar} gives $C_+=C_-$ for this family.

\section{A verified dyadic example}
\label{app:g1-dyadic}

We now verify the covariance hypothesis of Corollary~\ref{cor:dyadic} for a
concrete member of the balanced class.  By
Example~\ref{ex:symmetric-family}, the map $g_1$ below belongs to
$\mathfrak F_*$, has $\beta=1$, and has natural measure
$\nu_*=\tfrac12(\delta_{-1}+\delta_1)$.

\begin{lemma}
\label{lem:g1-critical-renewal}
Let $g=g_1$ be the map in Example~\ref{ex:symmetric-family}, and put
\[
  a=2-\sqrt3,
  \qquad Y=[-a,a].
\]
There is a symmetric, $g$-invariant, sigma-finite measure $\mu\sim m$,
uniquely normalized by $\mu(Y)=1$.  Let $L$ be the transfer operator with
respect to $\mu$, let
\[
  \tau(y)=\inf\{n\ge1:g^ny\in Y\},
  \qquad F=g^\tau:Y\to Y,
\]
where $\inf\varnothing=\infty$; the exceptional infinite-return set is
$\mu$-null and is omitted when defining $F$.  Define, on
$\mathcal B=\operatorname{BV}(Y)$,
\[
  R_nw=1_YL^n(1_{\{\tau=n\}}w),
  \qquad
  T_nw=1_YL^n(1_Yw).
\]
We use the norm
$\|w\|_{\mathcal B}=\|w\|_{L^1(\mu)}+\Var_Y(w)$ and, when $z$ is
complex, the complexification of this space.
Then
\begin{align}
  \mu(\tau>n)&\sim \frac{c_\tau}{n}
       &&\text{for some $c_\tau>0$,}                                  \label{eq:g1-return-tail}\\
  \|R_n\|_{\mathcal B\to\mathcal B}&\le C(n+1)^{-2},                 \label{eq:g1-Rn-bound}\\
  \|T_n\|_{\mathcal B\to\mathcal B}&\le \frac{C}{\log(n+2)}.       \label{eq:g1-Tn-bound}
\end{align}
Let $\mathcal Jw(x)=w(-x)$ and
$\mathcal B_{\rm o}=\{w\in\mathcal B:\mathcal Jw=-w\}$.  For every
$\eta\in(0,1)$,
\begin{equation}
  \|T_n|_{\mathcal B_{\rm o}}\|_{\mathcal B\to\mathcal B}
  \le C_\eta(n+1)^{-\eta}.                                             \label{eq:g1-odd-renewal}
\end{equation}
Moreover, if $w\in\mathcal B_{\rm o}$, then
\begin{equation}
  \|L^n(1_Yw)\|_{L^1(\mu)}
  \le C(n+1)^{-1/2}\|w\|_{\mathcal B}.                                \label{eq:g1-odd-global}
\end{equation}

Let
\[
  \kappa(x)=\inf\{n\ge0:g^nx\in Y\},
  \qquad v_0=\frac{dm}{d\mu},
\]
again with $\inf\varnothing=\infty$.
Set
\[
  b_c=1_YL^c(1_{\{\kappa=c\}}v_0),
  \qquad
  U_r=1_YL^rv_0.
\]
Then
\begin{align}
  \|b_c\|_{\mathcal B}&\le C(c+1)^{-2},
  &m(\kappa>c)&\le C(c+1)^{-1},                                        \label{eq:g1-entrance-bounds}\\
  \|U_r\|_{\mathcal B}&\le \frac{C}{\log(r+2)},
  &m(g^{-r}Y)&\le \frac{C}{\log(r+2)}.                                 \label{eq:g1-base-occupation}
\end{align}
\end{lemma}

\begin{proof}
The choice of $a$ is characterized by
\[
  g(a)=-a,
  \qquad g(-a)=a.
\]
Thus $Y$ is exactly the symmetric period-two inducing interval in
\eqref{eq:CMT-partition}.  Its first-return atoms are the intervals
$Y_{\pm,n}$, $n\ge2$, and $g^n:Y_{\pm,n}\to Y$ is a bijection.  In
particular, $F$ is a mixing full-branch Gibbs--Markov map and return times
$2$ and $3$ both occur.  We record the additional quantitative estimates
needed below, since the signed estimates do not follow merely from the
existence of this inducing scheme.

On the positive branch, use the endpoint coordinate $u=1-x$.  As long as
the orbit remains on that branch,
\[
  1-g(1-u)=q(u),
  \qquad q(u)=u+u^2.
\]
For $t$ in the fixed compact interval
$J=[1-a,1+a]\subset(0,2)$, write
\[
  u_n(t)=q^{-n}(t),
  \qquad u_0(t)=t.
\]
Since $u_{k-1}=u_k(1+u_k)$,
\begin{equation}
  \frac1{u_k}-\frac1{u_{k-1}}=\frac1{1+u_k}.                            \label{eq:g1-reciprocal-increment}
\end{equation}
There is $c_0<1$ such that $0<u_k(t)\le c_0$ for every $k\ge1$ and
$t\in J$.
Summing \eqref{eq:g1-reciprocal-increment} therefore gives, uniformly for
$t\in J$,
\begin{equation}
  u_n(t)\asymp(n+1)^{-1}.                                                \label{eq:g1-un-size}
\end{equation}
Differentiation gives
\[
  u_n'(t)=\prod_{k=1}^n(1+2u_k(t))^{-1}.
\]
Now
\[
  \prod_{k=1}^n(1+u_k)=\frac{t}{u_n}
\]
and
\[
  \frac{1+2u_k}{(1+u_k)^2}
  =1-\frac{u_k^2}{(1+u_k)^2}.
\]
By \eqref{eq:g1-un-size}, the sum of the terms $u_k^2$ is bounded
uniformly in $t\in J$.  Hence the product of the last display is bounded
above and away from zero, uniformly in $n$ and $t$.  It follows that
\begin{equation}
  |u_n'(t)|\asymp(n+1)^{-2}.                                             \label{eq:g1-un-prime}
\end{equation}
Finally,
\[
  \frac{u_n''(t)}{u_n'(t)}
  =-2\sum_{k=1}^n\frac{u_k'(t)}{1+2u_k(t)}.
\]
The right-hand side is uniformly bounded by
\eqref{eq:g1-un-prime}, and consequently
\begin{equation}
  |u_n''(t)|\le C(n+1)^{-2}.                                             \label{eq:g1-un-second}
\end{equation}

The two length-$n$ first-entrance inverse branches
$v_{n,\epsilon}:Y\to\{\kappa=n\}\subset I\setminus Y$ are given by
$y\mapsto1-u_n(1-y)$ and its reflection.  For $n\ge2$, a
length-$n$ first-return inverse branch is obtained by composing a
length-$(n-1)$ first-entrance inverse with the inverse of the opposite
branch of $g$.  That additional inverse has uniformly bounded first and
second derivatives.  Thus every first-entrance branch of length $n$ and
every first-return branch of length $n$ has inverse $v$ that
satisfies
\begin{equation}
  \|v'\|_\infty+\Var_Y(v')\le C(n+1)^{-2}.                              \label{eq:g1-inverse-branch-BV}
\end{equation}
There are two such branches of each sufficiently large length; the finitely
many remaining lengths are absorbed by increasing $C$.

Let $h=d\mu/d\Leb$ on $Y$.  The inducing argument in the proof of
Lemma~\ref{lem:alpha-one-CMT} gives a Lipschitz representative of $h$ that
is bounded above and away from zero.  If $v_{n,\epsilon}:Y\to
\{\tau=n\}$, $\epsilon\in\{+,-\}$, are the return inverse branches for
$n\ge2$, then
\begin{equation}
  R_nw
  =\sum_{\epsilon\in\{+,-\}}
    p_{n,\epsilon}\,w\circ v_{n,\epsilon},
  \qquad
  p_{n,\epsilon}
  =\frac{h\circ v_{n,\epsilon}}{h}\,|v_{n,\epsilon}'|.                 \label{eq:g1-return-branch-formula}
\end{equation}
Equations~\eqref{eq:g1-inverse-branch-BV} and
\eqref{eq:g1-return-branch-formula}, together with the bounds on $h$ and
$1/h$, give
\[
  \|p_{n,\epsilon}\|_{\operatorname{BV}(Y)}\le C(n+1)^{-2}.
\]
Since
$\Var_Y(w\circ v_{n,\epsilon})
 \le\Var_{\{\tau=n\}\cap Y_\epsilon}(w)
 \le\Var_Y(w)$,
the product rule for variation yields
\[
  \|p_{n,\epsilon}w\circ v_{n,\epsilon}\|_{\mathcal B}
  \le C(n+1)^{-2}\|w\|_{\mathcal B}.
\]
This proves \eqref{eq:g1-Rn-bound}.  More precisely,
\cite[Lemma~4.11]{CMT26}, with the index shift explained in the proof of
Lemma~\ref{lem:alpha-one-CMT}, gives
\[
  \Leb(Y_{-,n})\sim B_+n^{-2},
  \qquad
  \Leb(Y_{+,n})\sim B_-n^{-2}.
\]
Both intervals shrink to $0$, from the left and the right respectively.
Since $h$ is continuous at $0$ and $h(0)>0$,
\[
  \mu(\tau=n)
  =\int_{Y_{-,n}\cup Y_{+,n}}h\,d\Leb
  \sim h(0)(B_++B_-)n^{-2}
  =:c_\tau n^{-2}.
\]
Direct summation gives
\eqref{eq:g1-return-tail}.  In particular,
\begin{equation}
  \mu(\tau=n)\asymp n^{-2},                                              \label{eq:g1-return-atom-size}
\end{equation}
so \eqref{eq:g1-Rn-bound} also has the form
$\|R_n\|\le C\mu(\tau=n)$, after enlarging $C$ for the finitely many
small $n\ge2$.  This is operator-renewal hypothesis \textup{(H1)}; here
$R_1=0$ because $\tau\ge2$ almost everywhere.

The hypotheses (H1)--(H2) of the corrected operator-renewal theorem of
Melbourne--Terhesiu hold on $\mathcal B$.  Indeed,
\eqref{eq:g1-Rn-bound} and \eqref{eq:g1-return-atom-size} give (H1).
We verify the quasicompactness required in (H2).  Let $\alpha_k$ denote
the partition into $k$-cylinders for $F$, let $h_A:Y\to A$ be the inverse
branch corresponding to $A\in\alpha_k$, and let $p_A$ be its Jacobian with
respect to $\mu$.  Write
$\tau_k=\sum_{j=0}^{k-1}\tau\circ F^j$.  Uniform expansion and bounded
distortion give constants
$C<\infty$ and $\Lambda>1$ such that
\[
  \Var_Y(p_A)\le C\sup_Yp_A,
  \qquad
  \sup_Yp_A\le C\mu(A)\le C\Lambda^{-k}.
\]
Since
\[
  \sup_A|v|
  \le \mu(A)^{-1}\int_A|v|\,d\mu+C\Var_A(v),
\]
the product rule for variation, summed over the disjoint cylinders $A$,
gives, uniformly for $|z|\le1$,
\begin{equation}
  \Var_Y(\mathcal R(z)^kv)
  \le C\Lambda^{-k}\Var_Y(v)+C\|v\|_1,
  \qquad
  \|\mathcal R(z)^kv\|_1\le\|v\|_1,                                 \label{eq:g1-uniform-LY}
\end{equation}
because the twisting factor $z^{\tau_k}$ is constant on each
$k$-cylinder and has modulus at most one.  Thus, for some
$\vartheta\in(0,1)$,
\[
  \|\mathcal R(z)^kv\|_{\mathcal B}
  \le C\vartheta^k\|v\|_{\mathcal B}+C\|v\|_1,
  \qquad
  \mathcal R(z)=\sum_{n\ge1}z^nR_n.
\]
The compact embedding $\operatorname{BV}(Y)\hookrightarrow L^1(Y,\mu)$,
together with the preceding Doeblin--Fortet inequality, gives by the
Ionescu--Tulcea--Marinescu compactness theorem that every $\mathcal R(z)$ is
quasicompact on $\mathcal B$, with essential spectral radius strictly smaller
than $1$.  The eigenvalue
$1$ of
$\mathcal R(1)=\sum_{n\ge1}R_n$, the transfer operator of $F$, is simple
and isolated because $F$ is mixing and Gibbs--Markov.
If $|z|<1$, positivity and $\tau\ge2$ imply that a hypothetical eigenvector
$\mathcal R(z)v=v$ would satisfy
\[
  \|v\|_1\le |z|^2\|v\|_1,
\]
which is impossible for $v\ne0$.  Quasicompactness then excludes $1$ from
the spectrum.  It remains to consider $|z|=1$.  If $1$ belonged to the
spectrum of $\mathcal R(z)$, quasicompactness would make it an eigenvalue.
For a nonzero eigenvector $v\in\mathcal B$, the embeddings
$\mathcal B\hookrightarrow L^\infty(Y,\mu)\hookrightarrow L^2(Y,\mu)$ allow
us to apply the standard adjoint argument in
\cite[Proposition~11.1]{MT12}; it gives the almost-everywhere coboundary
relation
\[
  v\circ F=z^\tau v,
\]
Taking absolute values and using the ergodicity of $F$ shows that $|v|$ is a
positive constant almost everywhere.

Let $h_n:Y\to Y_{\epsilon,n}$ be either inverse branch of return length
$n$.  Fix the regulated representative of $v\in\operatorname{BV}(Y)$ that
has its canonical one-sided limits.  Each $h_n^j$ is a nonsingular
$C^1$ diffeomorphism onto its image.  Hence the union, over $j\ge0$, of the
pullbacks under $h_n^j$ of the null set where either the coboundary identity
or the chosen representative fails is still null.  Choose $y$ outside this
union.  Since $F\circ h_n=\operatorname{id}$ and
$\tau\circ h_n=n$, iteration gives
\[
  v(h_n^ky)=z^{-nk}v(y)
  \qquad(k\ge1).
\]
The inverse branch $h_n$ is increasing and contracting, so $h_n^ky$
converges monotonically to its fixed point.  A $BV$ representative of $v$
has the corresponding one-sided limit there.  Since $|v(y)|>0$, the last
display can converge only if $z^n=1$.  Applying this argument to return
branches of lengths $2$ and $3$ gives $z^2=z^3=1$, hence $z=1$.  At $z=1$
the simple eigenvalue already identified is the sole obstruction.  This
proves \textup{(H2)}.  Write
\[
  \mu(\tau>n)=\ell(n)n^{-1};
\]
by \eqref{eq:g1-return-tail}, $\ell(n)\to c_\tau$.  The corrected form of
\cite[Theorem~2.1]{MT12} states at exponent one that
\[
  \left(\sum_{j=1}^n\frac{\ell(j)}j\right)T_n
  \longrightarrow P
  \quad\text{in operator norm}
\]
for the rank-one spectral projection $P$.  Since the sum in parentheses is
asymptotic to $c_\tau\log n$, this proves
\eqref{eq:g1-Tn-bound}.

We next prove the odd estimate; this step is not part of the quoted renewal
theorem.  Put
\[
  \mathcal R(z)=\sum_{n\ge1}R_nz^n,
  \qquad
  \mathcal T(z)=\sum_{n\ge0}T_nz^n.
\]
For $|z|<1$, the renewal identity is
\begin{equation}
  \mathcal T(z)=(I-\mathcal R(z))^{-1}.                                  \label{eq:g1-renewal-identity}
\end{equation}
Reflection preserves $Y$ and $\tau$, and $g(-x)=-g(x)$ away from the
irrelevant point $0$.  Uniqueness of $\mu$ after the normalization
$\mu(Y)=1$ shows that $\mu$ is symmetric.  Hence $\mathcal J$ commutes with
$R_n$ and $T_n$.  At $z=1$, the only $1$-eigenfunctions of
$\mathcal R(1)$ are constant, and hence even.  Quasicompactness and the
simplicity of that eigenvalue therefore imply that
$I-\mathcal R(1)$ is invertible on $\mathcal B_{\rm o}$.  At every other
$z$ with $|z|=1$, invertibility follows from (H2); for $|z|<1$ it follows
from \eqref{eq:g1-renewal-identity}.  Consequently
\[
  \mathcal T_{\rm o}(z)
  :=(I-\mathcal R(z)|_{\mathcal B_{\rm o}})^{-1}
\]
extends continuously to the closed unit disc and is uniformly bounded
there.

Fix $\eta\in(0,1)$.  From \eqref{eq:g1-Rn-bound},
\[
  \sum_{n\ge1}n^\eta\|R_n\|<\infty.
\]
It follows that $t\mapsto\mathcal R(e^{it})$ is $\eta$-H\"older in
operator norm.  The resolvent identity gives the same property for
$t\mapsto\mathcal T_{\rm o}(e^{it})$.  Cauchy's formula on $|z|=\rho<1$
and uniform convergence as $\rho\uparrow1$ show that its $n$th boundary
Fourier coefficient is $T_n|_{\mathcal B_{\rm o}}$.  For $n\ge1$, taking
the difference of the Fourier integral with its translate by $\pi/n$ gives
\[
  2\|T_n|_{\mathcal B_{\rm o}}\|
  \le \frac1{2\pi}\int_{-\pi}^{\pi}
       \|\mathcal T_{\rm o}(e^{i(t+\pi/n)})
          -\mathcal T_{\rm o}(e^{it})\|\,dt
  \le C_\eta n^{-\eta}.
\]
The case $n=0$ is absorbed into the constant, proving
\eqref{eq:g1-odd-renewal}.

For $q\ge0$, define
\[
  A_qw=L^q(1_{\{\tau>q\}}w),
  \qquad w\in\mathcal B,
\]
where $w$ is extended by zero outside $Y$.  Partitioning an orbit segment
according to its last visit to $Y$ gives the exact identity
\begin{equation}
  L^n(1_Yw)=\sum_{j=0}^n A_{n-j}T_jw.                                   \label{eq:g1-last-return-identity}
\end{equation}
Here is a duality verification, which also fixes the endpoint conventions in
this decomposition.  For $0\le j\le n$, set
\[
 E_j^{(n)}=\{x\in Y:g^jx\in Y,
                    \ \tau(g^jx)>n-j\}.
\]
Modulo null sets, the $E_j^{(n)}$ are disjoint and cover $Y$: the index $j$
is the last visit to $Y$ during the time interval $[0,n]$ (and $j=0$ is
always available because the orbit starts in $Y$).  For every bounded
measurable test function $\psi$ on $I$, transfer-operator duality applied
first for $n-j$ steps and then for $j$ steps gives
\begin{align*}
 \int_I\psi\,A_{n-j}T_jw\,d\mu
 &=\int_Y
    1_{\{\tau>n-j\}}(y)\psi(g^{n-j}y)T_jw(y)\,d\mu(y)\\
 &=\int_Y1_{E_j^{(n)}}(x)\psi(g^nx)w(x)\,d\mu(x).
\end{align*}
Summing over $j$ gives the dual pairing with
$L^n(1_Yw)$, proving \eqref{eq:g1-last-return-identity} in $L^1(\mu)$.
Moreover, by \eqref{eq:g1-return-tail} and the embedding
$\operatorname{BV}(Y)\hookrightarrow L^\infty(Y)$,
\begin{equation}
  \|A_qw\|_1
  \le \int_{\{\tau>q\}}|w|\,d\mu
  \le C(q+1)^{-1}\|w\|_{\mathcal B}.                                  \label{eq:g1-Aq-bound}
\end{equation}
If $w$ is odd, then every $T_jw$ is odd.  Hence
\eqref{eq:g1-odd-renewal}--\eqref{eq:g1-Aq-bound} imply
\[
  \|L^n(1_Yw)\|_1
  \le C_\eta\sum_{j=0}^n
       (n-j+1)^{-1}(j+1)^{-\eta}\|w\|_{\mathcal B}
  \le C_\eta(n+1)^{-\eta}\log(n+2)\|w\|_{\mathcal B}.
\]
Taking $\eta=3/4$ and enlarging the constant proves
\eqref{eq:g1-odd-global}.

It remains to prove the entrance assertions.  The first-entrance inverse
branches and \eqref{eq:g1-inverse-branch-BV} give, for $c\ge1$,
\[
  b_c(y)=\frac1{2h(y)}
    \sum_{\epsilon\in\{+,-\}}|v_{c,\epsilon}'(y)|,
\]
where now $v_{c,\epsilon}$ denotes the inverse of the corresponding
first-entrance branch.  Thus
$\|b_c\|_{\mathcal B}\le C(c+1)^{-2}$; the case $c=0$ follows from the
regularity and positivity of $h$ on $Y$.  The points that have not entered
$Y$ by time $c$ form two endpoint intervals of length $O((c+1)^{-1})$ by
\eqref{eq:g1-un-size}.  This proves \eqref{eq:g1-entrance-bounds}.

Partitioning an initial orbit according to its first entrance into $Y$
gives
\begin{equation}
  U_r=\sum_{c=0}^rT_{r-c}b_c.                                            \label{eq:g1-first-entrance-renewal}
\end{equation}
Indeed, if $\psi$ is bounded on $Y$, then the definition of $b_c$ and two
applications of transfer-operator duality give
\[
 \int_Y\psi\,T_{r-c}b_c\,d\mu
 =\int_I1_{\{\kappa=c\}}(x)1_Y(g^rx)
      \psi(g^rx)v_0(x)\,d\mu(x).
\]
The sets $\{\kappa=c\}$, $0\le c\le r$, partition
$\{x:g^rx\in Y\}$ modulo null sets.  Summation therefore gives
$\int_Y\psi U_r\,d\mu$ and proves
\eqref{eq:g1-first-entrance-renewal} in $L^1(\mu)$.
Equations~\eqref{eq:g1-Tn-bound}, \eqref{eq:g1-entrance-bounds}, and a
split at $c=r/2$ yield
\[
  \|U_r\|_{\mathcal B}
  \le C\sum_{c=0}^r
       \frac{(c+1)^{-2}}{\log(r-c+2)}
  \le \frac{C}{\log(r+2)}.
\]
Finally,
\[
  m(g^{-r}Y)=\int_YU_r\,d\mu
  \le \|U_r\|_\infty\mu(Y),
\]
which completes the proof of \eqref{eq:g1-base-occupation}.
\end{proof}

\begin{lemma}
\label{lem:g1-signed-first-hit}
In the notation of Lemma~\ref{lem:g1-critical-renewal}, let
\[
  \sigma(x)=
  \begin{cases}
    -1,&x<0,\\
    1,&x>0,
  \end{cases}
\]
with an arbitrary value at $0$.  Put $v_r=L^rv_0$ and $q_r=\sigma v_r$.
For $\ell\ge0$, define
\[
  w_{r,\ell}
  =1_YL^\ell(1_{\{\kappa=\ell\}}q_r).
\]
Then $w_{r,\ell}\in\mathcal B_{\rm o}$ and
\begin{equation}
  \sup_{r\ge0}\sum_{\ell\ge0}\|w_{r,\ell}\|_{\mathcal B}<\infty.     \label{eq:g1-first-hit-BV-sum}
\end{equation}
If $P_r=(g^r)_*m$ and the integers $K,r\ge0$ satisfy
$K+1\ge(r+1)/2$, then
\begin{equation}
  P_r(\kappa>K)\le \frac{C}{\log(r+2)}.                                 \label{eq:g1-forward-residual}
\end{equation}
Consequently, for all integers $k\ge r\ge0$,
\begin{equation}
  \|L^kq_r\|_{L^1(\mu)}
  \le C\left\{\frac1{\log(r+2)}+(k+1)^{-1/2}\right\}.                \label{eq:g1-signed-transfer-bound}
\end{equation}
\end{lemma}

\begin{proof}
Both $m$ and $\mu$ are symmetric, and $g$ commutes with reflection modulo
the null orbit of the discontinuity.  Hence $v_r$ is even, $q_r$ is odd,
and the reflection invariance of $Y$ and $\kappa$ shows that every
$w_{r,\ell}$ is odd.

We prove \eqref{eq:g1-first-hit-BV-sum} without applying distortion to an
arbitrary measurable cut.  For $N\ge1$ and $0\le d<N$, define the marked
complete-return operator
\[
  R_{N,d}^{\sigma}w
  =1_YL^N(1_{\{\tau=N\}}(\sigma\circ g^d)w).
\]
For $N\ge2$, if a return starts in the negative half of $Y$, its sign is
negative at age zero and positive at every age $1,\ldots,N-1$; for the
reflected branch the signs are reversed.  Thus $\sigma\circ g^d$ is
constant on each of the two complete return branches.  For $N=1$ the
operator is zero.  The branch calculation in
\eqref{eq:g1-return-branch-formula} therefore gives
\begin{equation}
  \|R_{N,d}^{\sigma}w\|_{\mathcal B}
  \le C(N+1)^{-2}\|w\|_{\mathcal B}
  \qquad(0\le d<N).                                                       \label{eq:g1-marked-return-bound}
\end{equation}

For $\ell\ge1$, let
\[
  d_{r,\ell}
  =1_YL^{r+\ell}
      \bigl(1_{\{\kappa=r+\ell\}}(\sigma\circ g^r)v_0\bigr).
\]
On either complete first-entrance branch of length $r+\ell$, the orbit has
constant sign at each age smaller than $r+\ell$.  The first-entrance version
of \eqref{eq:g1-inverse-branch-BV} hence gives
\begin{equation}
  \|d_{r,\ell}\|_{\mathcal B}
  \le C(r+\ell+1)^{-2}.                                                   \label{eq:g1-marked-entrance-bound}
\end{equation}

Consider an initial orbit whose position at time $r$ first reaches $Y$
after a further $\ell\ge1$ iterates.  Either the initial orbit has not
visited $Y$ by time $r$, or it has a unique last visit at some
$j\in\{0,\ldots,r-1\}$.  In the second case its next return length is
$r-j+\ell$, and its age at time $r$ is $r-j$.  This disjoint path
decomposition gives the exact density identity
\begin{equation}
  w_{r,\ell}
  =d_{r,\ell}
   +\sum_{j=0}^{r-1}
       R_{r-j+\ell,r-j}^{\sigma}U_j,
  \qquad \ell\ge1.                                                       \label{eq:g1-marked-last-visit}
\end{equation}
To verify the identity, test both sides against a bounded measurable
$\psi:Y\to\mathbb C$.  By duality, the pairing with the left-hand side is
\begin{equation}
 \int_I (\sigma\circ g^r)(x)\,
   1_{\{\kappa(g^rx)=\ell\}}
   \psi(g^{r+\ell}x)v_0(x)\,d\mu(x).
 \label{eq:g1-marked-left-duality}
\end{equation}
The pairing with $d_{r,\ell}$ is the restriction of this integral to
$\{\kappa(x)=r+\ell\}$, namely to orbits having no visit to $Y$ before time
$r+\ell$.  For the term indexed by $j$, duality first through $U_j$ and then
through $R_{r-j+\ell,r-j}^{\sigma}$ gives the restriction to
\[
 \{x:g^jx\in Y,
       \ \tau(g^jx)=r-j+\ell\}.
\]
On this set the mark is
$\sigma(g^{j+(r-j)}x)=\sigma(g^rx)$ and the endpoint is
$g^{j+(r-j+\ell)}x=g^{r+\ell}x$.  The event in
\eqref{eq:g1-marked-left-duality} is the disjoint union of the no-previous-
visit event and these events indexed by the unique last visit
$j\in\{0,\ldots,r-1\}$.  Equality of the pairings for every $\psi$ proves
\eqref{eq:g1-marked-last-visit} in $L^1(\mu)$.
The upper limit is $r-1$: a point in $Y$ at time $r$ has $\kappa=0$.
Also
\[
  w_{r,0}=1_Yq_r=\sigma U_r.
\]
Multiplication by $\sigma$ is bounded on $\operatorname{BV}(Y)$, since it
introduces only one jump.  Using \eqref{eq:g1-base-occupation},
\eqref{eq:g1-marked-return-bound},
\eqref{eq:g1-marked-entrance-bound}, and
\eqref{eq:g1-marked-last-visit}, we obtain
\begin{align*}
  \sum_{\ell\ge0}\|w_{r,\ell}\|_{\mathcal B}
  &\le C+C\sum_{j=0}^{r-1}
       \frac1{\log(j+2)}
       \sum_{\ell\ge1}(r-j+\ell+1)^{-2}\\
  &\le C+C\sum_{j=0}^{r-1}
       \frac1{\log(j+2)(r-j+1)}
  \le C.
\end{align*}
For the last inequality, split the sum at $j=r/2$.  On the first part use
$(r-j+1)^{-1}\le C(r+1)^{-1}$, and on the second use
$1/\log(j+2)\le C/\log(r+2)$ and the harmonic-sum bound.  This proves
\eqref{eq:g1-first-hit-BV-sum}.

The same disjoint last-visit decomposition, now without the sign, gives
the exact probability identity
\begin{equation}
  P_r(\kappa>K)
  =m(\kappa>r+K)
   +\sum_{j=0}^{r-1}
      \int_Y1_{\{\tau>r-j+K\}}U_j\,d\mu.                               \label{eq:g1-residual-identity}
\end{equation}
Indeed, transfer-operator duality identifies the summand indexed by $j$ with
\[
 m\{x:g^jx\in Y,\ \tau(g^jx)>r-j+K\}.
\]
The event $\{\kappa(g^rx)>K\}$ is, modulo null sets, the disjoint union of
$\{\kappa(x)>r+K\}$ and the displayed events: in the latter case $j$ is the
unique last visit to $Y$ no later than time $r-1$, and the return following
that visit occurs more than $K$ iterates after time $r$.  This proves
\eqref{eq:g1-residual-identity}.  By
\eqref{eq:g1-return-tail}, \eqref{eq:g1-entrance-bounds}, and
\eqref{eq:g1-base-occupation},
\[
  P_r(\kappa>K)
  \le \frac{C}{r+K+1}
    +C\sum_{j=0}^{r-1}
       \frac1{\log(j+2)(r-j+K+1)}.
\]
If $K+1\ge(r+1)/2$, then the last denominator is bounded below by a
constant multiple of $r+1$, while
\[
  \sum_{j=0}^{r-1}\frac1{\log(j+2)}
  \le \frac{C(r+1)}{\log(r+2)}.
\]
This proves \eqref{eq:g1-forward-residual}.

Finally, first entrance from the position at time $r$ gives
\begin{equation}
  L^kq_r
  =L^k(1_{\{\kappa>k\}}q_r)
   +\sum_{\ell=0}^kL^{k-\ell}w_{r,\ell}.                                \label{eq:g1-current-first-hit}
\end{equation}
For clarity, this identity also follows directly by duality.  The disjoint
partition
\[
 I=\{\kappa>k\}\mathbin{\dot\cup}
     \bigcup_{\ell=0}^k\{\kappa=\ell\}
 \quad\text{modulo the null set }\{\kappa=\infty\}
\]
is applied to the density $q_r$.  On $\{\kappa=\ell\}$, transfer for the
first $\ell$ iterates produces $w_{r,\ell}$ on $Y$, and transfer for the
remaining $k-\ell$ iterates produces
$L^{k-\ell}w_{r,\ell}$.  Testing against arbitrary bounded functions proves
\eqref{eq:g1-current-first-hit} in $L^1(\mu)$.
Put $K=\lfloor k/2\rfloor$.  For $\ell\le K$, apply
\eqref{eq:g1-odd-global} and then
\eqref{eq:g1-first-hit-BV-sum}.  For $\ell>K$, use $L^1$ contraction and
the disjointness of the first-hit sets.  The first term in
\eqref{eq:g1-current-first-hit} has $L^1$ norm at most
$P_r(\kappa>k)$.  We obtain
\begin{align*}
  \|L^kq_r\|_1
  &\le C(k+1)^{-1/2}
       \sum_{\ell=0}^K\|w_{r,\ell}\|_{\mathcal B}
       +P_r(\kappa>k)
       +\sum_{\ell>K}\|w_{r,\ell}\|_1\\
  &\le C(k+1)^{-1/2}+2P_r(\kappa>K).
\end{align*}
When $k\ge r$, one has $K+1\ge(r+1)/2$, so
\eqref{eq:g1-forward-residual} proves
\eqref{eq:g1-signed-transfer-bound}.
\end{proof}

\begin{proposition}
\label{prop:g1-dyadic-example}
For the symmetric quadratic map $g_1\in\mathfrak F_*$, namely the $p=1$
member of Example~\ref{ex:symmetric-family},
\[
  g_1(x)=
  \begin{cases}
    x+(1+x)^2,&-1\le x\le0,\\
    x-(1-x)^2,&0<x\le1,
  \end{cases}
\]
there is a countable uniformly dense set
$\mathcal D\subset C([-1,1])$ such that, for every
$\phi\in\mathcal D$, there is $C_\phi<\infty$ satisfying
\begin{equation}
  \sum_{0\le i,j<n}
  \left|
    \Cov_m\bigl(\phi\circ g_1^{2^i},\phi\circ g_1^{2^j}\bigr)
  \right|
  \le C_\phi n\log(n+2)                                                  \label{eq:g1-dyadic-double-sum}
\end{equation}
for every $n\ge1$.  In particular, the hypothesis of
Corollary~\ref{cor:dyadic} holds for $g_1$, for example with
$\delta_\phi=1/2$ for every $\phi\in\mathcal D$.  Consequently,
\[
  \mu_n^{\mathrm{dyad}}(x)\weakstar
  \frac12\delta_{-1}+\frac12\delta_1
\]
for $m$-almost every $x$, and $m(X_{\mathrm{dyad}})=0$.
\end{proposition}

\begin{proof}
Let $\mathcal D$ be a countable uniformly dense family of continuous
piecewise-linear functions with rational breakpoints and rational values
such that every member is constant on a neighborhood of each endpoint.
Such a family is obtained by first making a continuous function constant on
sufficiently short endpoint intervals and then applying rational
piecewise-linear approximation.

We first estimate the sign observable.  Symmetry gives
\[
  \int\sigma\circ g^t\,dm=0
  \qquad(t\ge0).
\]
For $s=r+k$, transfer-operator duality and the notation of
Lemma~\ref{lem:g1-signed-first-hit} give
\begin{align}
  \left|
    \Cov_m(\sigma\circ g^r,\sigma\circ g^{r+k})
  \right|
  &=\left|\int_I\sigma\,L^k(\sigma v_r)\,d\mu\right|
    \notag\\
  &\le \|L^kq_r\|_1                                                     \notag\\
  &\le C\left\{\frac1{\log(r+2)}+(k+1)^{-1/2}\right\},                \label{eq:g1-sign-covariance}
\end{align}
whenever $k\ge r$.

Fix $\phi\in\mathcal D$ and write
\[
  A_\phi=\frac{\phi(1)+\phi(-1)}2,
  \qquad
  B_\phi=\frac{\phi(1)-\phi(-1)}2,
  \qquad
  \psi=\phi-A_\phi-B_\phi\sigma.
\]
The function $\psi$ is bounded and vanishes on neighborhoods of both
endpoints.  Hence its support is contained in a compact set
$K_\phi\subset(-1,1)$.  Such a compact set is contained in $Y$ together
with finitely many of the Markov levels $X_{\pm,q}$ from
\eqref{eq:CMT-partition}.  Since $g^qX_{\pm,q}\subset Y$, there is an
integer $N_\phi$ such that
\begin{align}
  m(g^{-r}K_\phi)
  &\le m(g^{-r}Y)
      +\sum_{q=1}^{N_\phi}\sum_{\epsilon\in\{+,-\}}
          m(g^{-r}X_{\epsilon,q})                                       \notag\\
  &\le m(g^{-r}Y)+2\sum_{q=1}^{N_\phi}m(g^{-(r+q)}Y)                     \notag\\
  &\le \frac{C_\phi}{\log(r+2)}.                                       \label{eq:g1-compact-occupation}
\end{align}

Also,
\[
  \left|\int\psi\circ g^t\,dm\right|
  \le \|\psi\|_\infty m(g^{-t}K_\phi).
\]
Constants do not affect covariance.  Expanding
$\phi=A_\phi+B_\phi\sigma+\psi$, every term involving $\psi\circ g^r$
is bounded by a constant times $m(g^{-r}K_\phi)$, and every term involving
$\psi\circ g^s$ but not $\psi\circ g^r$ is bounded by a constant times
$m(g^{-s}K_\phi)$; the displayed estimate also controls the corresponding
products of means.  Thus
\eqref{eq:g1-sign-covariance} and
\eqref{eq:g1-compact-occupation} imply, for $s\ge2r\ge2$,
\begin{equation}
  \left|
    \Cov_m(\phi\circ g^r,\phi\circ g^s)
  \right|
  \le C_\phi\left\{
       \frac1{\log(r+2)}+(s-r+1)^{-1/2}
     \right\}.                                                          \label{eq:g1-general-covariance}
\end{equation}

Take $r=2^i$ and $s=2^j$ with $j>i$.  Then $s-r\ge r$, and
\eqref{eq:g1-general-covariance} gives
\[
  \left|
    \Cov_m(\phi\circ g^{2^i},\phi\circ g^{2^j})
  \right|
  \le C_\phi\left((i+1)^{-1}+2^{-i/2}\right).
\]
The diagonal terms contribute $O_\phi(n)$.  Therefore
\begin{align*}
  \sum_{0\le i,j<n}
  \left|
    \Cov_m(\phi\circ g^{2^i},\phi\circ g^{2^j})
  \right|
  &\le C_\phi n
    +2C_\phi\sum_{i=0}^{n-2}(n-i-1)
       \left((i+1)^{-1}+2^{-i/2}\right)\\
  &\le C_\phi n\log(n+2).
\end{align*}
This proves \eqref{eq:g1-dyadic-double-sum}.  Since
$n\log(n+2)\le Cn^{3/2}$, the dyadic covariance condition holds with
$\delta_\phi=1/2$.  The final assertions now follow from
Corollary~\ref{cor:dyadic} and the identity $\nu_*=\tfrac12\delta_{-1}
+\tfrac12\delta_1$ proved in Example~\ref{ex:symmetric-family}.
\end{proof}

\section*{Funding}
No funding was received for this work.

\section*{Declaration of competing interest}
The author declares that they have no known competing financial interests or
personal relationships that could have appeared to influence the work
reported in this paper.

\section*{Data availability}
No data were used for the research described in this article.

\end{document}